\documentclass[reqno]{amsart}
\usepackage[margin=3cm]{geometry}
\RequirePackage{etex}
\usepackage[utf8]{inputenc}
\usepackage[T1]{fontenc}
\usepackage[english]{babel}
\usepackage[square,sort,comma,numbers]{natbib}

\usepackage{scrtime} 
\usepackage{amsfonts}           
\usepackage{amscd}              
\usepackage{amsmath}            
\usepackage{amssymb,amsthm}     
\usepackage{xparse}             
\usepackage{xcolor}             
\usepackage{dsfont}             
\usepackage{bm, bbm}            
\usepackage{comment}            
\usepackage{multicol, bigints}
\usepackage[many]{tcolorbox}    
\usepackage{enumitem}           
\usepackage{float,graphicx,mathtools}              
\usepackage{array}              
\usepackage{tikz}               
\usetikzlibrary{arrows,positioning} 
\usetikzlibrary{patterns}       
\usepackage[normalem]{ulem}
\usepackage{mathrsfs}

\newcommand{\1}{\mathbbm 1}
\newcommand\C{\mathbb C}

\newcommand\N{\mathbb N}

\newcommand\R{\mathbb R}

\newcommand\cA{{\mathcal A}}
\newcommand\cB{{\mathcal B}}

\newcommand\cD{{\mathcal D}}
\newcommand\cF{{\mathcal F}}
\newcommand\cH{{\mathcal H}}
\newcommand\cK{{\mathcal K}}
\newcommand\cL{{\mathcal L}}

\newcommand\cP{{\mathcal P}}
\newcommand\cQ{{\mathcal Q}}
\newcommand\cR{{\mathcal R}}
\newcommand\cS{{\mathcal S}}
\newcommand\cT{{\mathcal T}}

\newcommand\cW{{\mathcal W}}

\newcommand\fA{{\mathfrak A}}
\newcommand\fp{{\mathfrak p}}

\providecommand{\norm}[1]{{\lVert#1\rVert}}

\providecommand{\norm}[1]{{\lVert#1\rVert}}
\newtheorem{theorem}{Theorem}[section]{\bf}{\it}
\newtheorem{proposition}[theorem]{Proposition}{\bf}{\it}
\newtheorem{corollary}[theorem]{Corollary}{\bf}{\it}
\newtheorem{lemma}[theorem]{Lemma}{\bf}{\it}
{\bf}{\it}
{\bf}{\it}
{\bf}{\it}
\newcounter{theoremintrocnt}
{\bf}{\it}

\newcommand{\quadtext}[1]{\quad\text{#1}\quad}
\theoremstyle{definition}
\newtheorem{definition}[theorem]{Definition}
\newtheorem{remark}[theorem]{Remark}
\newtheorem{example}[theorem]{Example}
\title[Finite dimensional approximation for linear operators]{
Finite dimensional approximations for Hilbert space operators and applications in Quantum Mechanics}
\author{Eva A. Gallardo-Guti\'errez}
\address{Departamento de An\'alisis Matem\'atico y Matem\'atica Aplicada,
Universidad Complutense de Madrid, 28040 Madrid, Spain and
  Instituto de Ciencias Matem\'aticas ICMAT (CSIC-UAM-UC3M-UCM), Madrid}
\email{eva.gallardo@mat.ucm.es}

\author{Fernando Lled\'o} %
\address{Department of Mathematics, University Carlos III de Madrid,
  Avda. de la Universidad 30, 28911. Legan\'es (Madrid), Spain and
  Instituto de Ciencias Matem\'aticas ICMAT (CSIC-UAM-UC3M-UCM), Madrid}
\email{flledo@math.uc3m.es}

\author{Laura S\'aenz} %
\address{Department of Mathematics, University Carlos III de Madrid,
  Avda. de la Universidad 30, 28911. Legan\'es (Madrid), Spain and
  Instituto de Ciencias Matem\'aticas ICMAT (CSIC-UAM-UC3M-UCM), Madrid}
\email{lsaenz@math.uc3m.es}

\thanks{This work is partially supported by the Severo Ochoa Centers of Excellence program through the project CEX2023-001347-S funded by MCIN/AEI/10.13039/501100011033.
First author is also  partially supported by Plan Nacional I+D grant no. PID2022-137294NB-I00, Spain. The second and third authors are  also partially 
supported by the grant PID2023-146758NB-I00 funded by MICIU/AEI/10.13039/501100011033.
The last named author also acknowledges support from the Ministry of Science, Innovation and Universities (grant No. FPU23/03378) and a postgraduate fellowship at the Residencia de Estudiantes, funded by the Ministry of Science and Innovation.}


\begin{document}

\begin{abstract}
In this work, we develop a unified framework for quasidiagonal and F\o lner-type approximations of linear operators on Hilbert spaces. These approximations (originally formulated for bounded operators and operator algebras) involve sequences of non-zero finite-rank orthogonal projections that asymptotically commute with the operator -- either in norm (quasidiagonal) or in mean 
(F\o lner). Such structures guarantee spectral approximation results in terms of their finite sections. We extend this theory to unbounded, densely defined closable operators, establishing a generalization of Halmos' classical result: every closable quasidiagonal operator is a compact perturbation of a closable block-diagonal operator on the same domain. Likewise, we introduce \textit{sparse F\o lner sequences} and establish an interplay between quasidiagonal approximations and the existence of sparse F\o lner sequences. The theoretical developments are illustrated with explicit examples using different types of weighted shifts and applied to quantum mechanical models, including a detailed treatment of the Weyl algebra and its Schr\"odinger representation.
\end{abstract}

\date{\today, \thistime,  \emph{File:} \texttt{\jobname.tex}}

\maketitle
\tableofcontents

\section{Introduction}
Approximation is one of the most important techniques in single operator theory and in operator algebras, allowing us to analyse complex objects in terms of limits of simpler ones (see, for example, \cite{Herrero82,bBrown08,boettcher:00}). Central to this line of inquiry are the concepts of 
\emph{quasidiagonal} and \emph{F\o lner approximations}. A bounded linear operator on a separable infinite dimensional Hilbert space
$T\in\cB(\cH)$ is quasidiagonal if there exists a sequence of finite-rank orthogonal projections 
$\cP=\{P_n\}_{n\in\N}$ which is increasing, exhaustive 
(i.e. converging strongly to the identity $\1$) and commuting asymptotically with $T$ in norm, i.e.
\[
\lim_{n\to\infty}\|\,TP_n  - P_n T\,\|=0\;.
\]
(In the case of a separable C*-algebra $\cA\subset\cB(\cH)$, one requires the preceding condition for any element of the algebra.) The corresponding sequence of projection subspaces is called a filtration of the Hilbert space $\cH$. Quasidiagonality was introduced by Halmos in the seventies in relation to the invariant subspace problem (see \cite{Halmos70,Voiculescu93}).
F\o lner approximations in terms of a sequence of finite-rank orthogonal projections $\cR=\{R_n\}_{n\in\N}$
generalize quasidiagonality in two ways. First,  the sequence of  projections is not required to be exhaustive, and 
second, the asymptotic condition is weaker in the sense that the limit is taken relative to the growth of the dimensions of the underlying projection subspaces. Concretely, $\cR$ is a F\o lner sequence for $T$ if
\[
\lim_{n\to\infty} \frac{\|T R_n-R_n T\|_1}{\|R_n\|_1} = 0\;,
\]
where $\|\cdot\|_1$ denotes the trace-class norm. The preceding condition can be equivalently stated with the Hilbert-Schmidt norm $\|\cdot\|_2$.
F\o lner-type approximations were introduced by Connes in his celebrated article 
\cite[Section~V]{Connes76}, particularly in his classification of injective von Neumann factors. Both the motivation for this approximation and its name stem from considering group algebras of amenable groups, i.e., groups admitting a F\o lner sequence of finite subsets in the group.
(See Section~\ref{sec:bounded} for precise statements as well as \cite{LledoYakubovich13,ALLY14} for additional results.) 

The {\em existence} of quasidiagonalizing or F\o lner sequences for an operator $T$ or a C*-algebra $\cA$ has several structural consequences. The most prominent one in single operator theory is due to Halmos and states that any quasidiagonal operator is a compact perturbation of a block-diagonal operator. As an immediate consequence, any quasidiagonal Fredholm operator necessarily has index $0$. This result extends a classic area of research in the analysis of self-adjoint and normal operators due to Weyl, von Neumann, Berg  and  Sikonia \cite{Weyl09,Neumann35,berg:71,Sikonia70}. In a similar vein, a separable C*-algebra $\cA\subset\cB(\cH)$ has a F\o lner sequence if and only if it admits an amenable trace $\tau$, i.e., a tracial state on $\cA$ that extends to a state $\psi$ on $\cB(\cH)$ that has $\cA$ in its centralizer. Recent results also provide characterizations of $\mathrm{C}^*$-algebras admitting amenable traces in terms of a matrix approximation given by a sequence of contractive completely positive maps which asymptotically commute in a normalized Hilbert-Schmidt norm, which makes contact with Voiculescu's modern C*-algebraic approach to quasidiagonality (see \cite{bBrown08,Brown06,Bedos95,AL14} for details).

Moreover, the \emph{explicit construction} of these sequences presents its own challenge. It allows numerical approximation of spectral objects of $T$ (spectral measures, pseudospectrum, spectrum, etc.) in terms of the corresponding quantities of its finite sections given by $P_nTP_n$. In fact,
if $T$ is quasidiagonal with respect to a filtration given by $\cP=\{P_n\}_{n\in\N}$, then there is in general convergence of pseudospectra, and in the normal case even better convergence results follow (see \cite{Brown06b,Arveson94} for details and more results). Quasidiagonal approximations and the finite section method have been considered for unbounded operators in \cite{Hansen08}.

In the more general context of F\o lner approximations there exist also interesting spectral approximations generalizing classical results 
obtained for Toeplitz operators by Szeg\"o in \cite{Szego20}.
Let $\cA\subset \cB(\cH)$ be a concrete C*-algebra (with 
$\cA\cap\cK(\cH)=\{0\}$)
admitting an amenable trace $\tau$ or, equivalently, a F\o lner sequence $\cR=\{R_n\}_{n\in\N}$ approximating $\tau$ (see \cite[Proposition~3.1]{AL14}). For $T=T^*\in\cA$ denote by $\mu_T$ the spectral measure associated with the trace
$\tau$ and consider the corresponding (self-adjoint)
compressions $T_n:=R_n T R_n$. Denote by $\mu_T^n$ the
probability measure on $\R$ supported on the spectrum of $T_n$,
i.e.,
\[
 \mu_T^n(\Delta):=\frac{N_T^n(\Delta)}{\|R_n\|_1}\;,\quad \Delta\subset\R\quad \mathrm{Borel}\;,
\]
where $N_T^n(\Delta)$ is the number of eigenvalues of $T_n$
(multiplicities counted) contained in $\Delta$. We say that 
$\left\{
\{R_n\}_{n\in\N} \,,\,\tau\right\}$ is a {\it Szeg\"o pair} for $\cA$ if for any self-adjoint element $T\in\cA$ we have
$\mu_T^n\mathop{\longrightarrow}\limits^{w^*} \mu_T$,
i.e., for any continuous function $f\colon \R\to\R$
\begin{equation}\label{eq:szego}
 \lim_{n\to\infty}
  \frac{1}{d_n}
  \Big( f(\lambda_{1,n})+\dots+f(\lambda_{d_n,n})\Big) =\int f(\lambda) \, d\mu_T(\lambda)
  \;,\quad f\in C_0(\R)  \;,
\end{equation}
where $d_n:=\|R_n\|_1$ and $\{\lambda_{1,n},\dots,\lambda_{d_n,n}\}$ are the eigenvalues
of $T_n$ repeated according to their multiplicity. 
Note that if the F\o lner sequence is not exhaustive
the approximated tracial state $\tau$ will not be faithful in general and the spectral approximation above 
may only involve part of the spectrum of the original operator $T$ (cf. \cite[Introduction]{Bedos97}).
We refer to \cite{Bedos97,bHagen01,AL14} for additional results and to \cite{pLledo11} for explicit 
F\o lner sequences in the context of tensor products and crossed products.

While almost all of the existing literature has focused on \emph{bounded} operators, many physically relevant operators such as the position operator $q$ or momentum operator $p$ satisfying the canonical commutation relation (in its simplest form and taking  the reduced Planck's constant $\hbar=1$)
 \begin{equation}\label{eq:ccr}
 qp-pq=i\1\;
\end{equation}
or standard quantum Hamiltonians are necessarily unbounded  (cf. \cite{Wintner47,Wielandt49}).
In quantum mechanics, \emph{Galerkin-type} approximations have been proposed as a method to approximate dynamics by restricting the problem to finite dimensional subspaces of increasing size, yet a systematic extension of F\o lner and quasidiagonal approximations to this setting remained underdeveloped.

In this article we address both the existence and construction of quasidiagonal and F\o lner sequences and their mutual relation for bounded and unbounded linear operators.
In the setting of unbounded operators we introduce in Section~\ref{sec:unbounded}
natural seminorms and require that the operator domain is stable under finite-rank projections. In this framework, we prove in Theorem~\ref{Theorem:QD=BD+K} a generalization of Halmos' theorem: any closable quasidiagonal operator is a compact perturbation of a block diagonal operator (on the same domain). Moreover, 
we construct a quasidiagonalizing sequence of projections associated with any (closable) normal operator based on Berg's method (see \cite{berg:71}) and a dyadic partition of the spectrum (see Subsections~\ref{subsec:berg} and \ref{subsec:berg-unbounded}).

Since the definition of F\o lner sequence does not assume exhaustivity of the finite-rank projections, we establish a natural interplay between quasidiagonality and F\o lner sequences. 
An extreme case of non-exhaustive F\o lner sequences are called sparse in Section~\ref{sec:sparse} (see Definition~\ref{def:gapped-sparse}) whose existence and construction we will apply for bounded and unbounded linear operators.
In this case the projections $R_n$ select a thin subset of a prefixed 
basis of the Hilbert space. 
Using the decomposition $T=B+K$ of a quasidiagonal operator into a block-diagonal operator and a compact one, we show that $T$ admits a sparse Følner sequence whenever the ranks of the blocks of $B$ do not grow too fast (see Theorem~\ref{Thm: gapped_sparse_QD} for bounded operators and Corollary~\ref{cor:usparse} for the unbounded case). A potential application of this technique is
model order reduction (MOR). 
These methods aim to compute a reduced-order model of a complex large-scale mathematical model that describes some physical or engineering process. In this context, one often
considers first the large-scale matrix by means of a discretization of a generically infinite dimensional system
via finite elements or finite differences of the typically unbounded differential operators describing the process. The model order reduction 
is then described by an orthogonal projection 
to a low-dimensional subspace carrying a significant information of the system.
In this sense sparse F\o lner sequences provide a substantial reduction of the original canonical finite sections (in the sense that the density vanishes asymptotically) but may retain Szeg\"o type spectral approximation result of the relevant part of the spectrum as in Eq.~\eqref{eq:szego} (cf. \cite{benner15}).

We will present along the way many concrete examples of operators on $\ell_2(\N)$, mainly in terms of unbounded weighted shifts. These examples will show, among other results, that for unbounded operators, a F\o lner sequence for the trace-class norm need not be a F\o lner sequence for the Hilbert-Schmidt norm. In these examples the propagation of the operator and the propagation of its commutator with the projection will play an important role.
Finally, we apply previous results to several mathematical structures associated to the canonical commutation relation (CCR): we show that the Weyl algebra, that is the complex algebra generated by 
$q$ and $p$ (see \cite{dixmier58,dixmier68}), satisfies an algebraic version of amenability introduced by Gromov in \cite{Gromov99}. 
This result confirms, at a purely algebraic level, that the canonical commutation relation generates amenable structures in the corresponding categories. In fact, it is shown in \cite[Theorems~3.1 and 4.1]{LM22} that every faithful and essential representation of Weyl C*-algebra and the resolvent C*-algebras (which encode a version of the commutation relation compatible with bounded operators) have F\o lner sequences of projections (see Definition~\ref{def:foelner} below and 
\cite{Manuceau68,BG08}).
We also apply our previous results, in particular the classes of examples mentioned before, to the Schr\"odinger representation of the Weyl algebra (necessarily) in terms of unbounded operators. We show in Theorem~\ref{Theorem: FolnerSequence_PQ} that, in contrast with the bounded case, the algebraic F\o lner subspaces do not induce F\o lner sequences of projections for the unbounded operator representing the generators of the Weyl algebra. Nevertheless, since these operators are essentially self-adjoint their closures do have
a different quasidiagonalizing sequence constructed in Subsection~\ref{subsec:berg-unbounded} that uses the corresponding spectral projections on a dyadic partition of their spectrum 
(see also Corollary~\ref{cor:qd-unbounded}).
\vspace{1mm}

\textbf{Notation:} We denote by $\cH$ a separable infinite dimensional complex Hilbert space.
The C*-algebras of bounded linear operators (resp. 
compact operators) on $\cH$ is denoted by $\cB(\cH)$ (resp. $\cK(\cH)$). Moreover, the set $\cL(\cH)$ corresponds to the set of linear (possibly unbounded) and densely defined operators on $\cH$.
For any $T\in\cK(\cH)$, we write $\|T\|_a$, $a=1,2,\dots$, for its norm in the Schatten-von Neumann class.

\section{Quasidiagonality and F\o lner sequences for bounded operators}\label{sec:bounded}
We begin by mentioning the main definitions of quasidiagonality and F\o lner sequences in the context of 
bounded operators on a Hilbert space. We also state some standard results needed later. For proofs and additional results we refer to \cite{Halmos70,Brown06,Brown06b,ALLY14} (see also references therein).

To motivate the notions introduced below recall that a bounded operator $T\in\cB(\cH)$ is called block-diagonal if there exists an exhausting
sequence $\{P_n\}_{n\in\N}$ of finite-rank orthogonal projections reducing $T$, i.e., satisfying $[T,P_n]=0$ for all $n\in\N$. The block structure is induced by the mutually orthogonal sequence of projections $Q_n:= P_n-P_{n-1}$ (putting $P_0:=0$) which satisfy $\sum_{n\in\N} Q_n=\1$ (convergence in the strong operator topology) and decompose the Hilbert space and the operator into 
\[
 \cH=\mathop{\oplus}\limits_{n\in\N} Q_n\cH\quadtext{and} 
 T=\sum_{n\in\N} Q_n T=\sum_{n\in\N} Q_n TQ_n\;.
\]
This notion has two natural generalizations:

\begin{definition}\label{def:qd}
Consider a filtration of the Hilbert space $\cH$, that is, a sequence of non-zero finite dimensional subspaces $\cH_1\subset\cH_2\subset\cdots$ such that $\overline{\cup_{n\in\N}\cH_n}=\cH$, and let $P_n$ be the finite-rank orthogonal projection onto $\cH_n$. For $T\in \cB(\cH)$, we say that the sequence
$\cP:=\{P_n\}_{n\in \N}$ is a {\em quasidiagonalizing sequence for $T$} if
\begin{equation}\label{eq:qd}
\lim_{n\to\infty} {\|T P_n-P_n T\|} = 0\;.
\end{equation}
An operator $T\in\cB(\cH)$ is called {\em quasidiagonal} if it admits a quasidiagonalizing sequence.
Similarly, a separable set of operators $\cT\subset\cB(\cH)$ (in particular, a concrete C*-algebra)
is {\em quasidiagonal} if there exists a sequence $\cP$ that is quasidiagonalizing for every element of $\cT$.
\end{definition}

\begin{definition}\label{def:foelner}
Let $T\in \cB(\cH)$ and $\cR:=\{R_n\}_{n\in \N}$ be a sequence of non-zero finite-rank orthogonal projections. The sequence $\cR$ is called a {\em F{\o}lner sequence for $T$} if
\begin{equation}\label{eq:foelner}
\lim_{n\to\infty} \frac{\|T R_n-R_n T\|_1}{\|R_n\|_1} = 0\;.
\end{equation}
If the previous condition holds for all elements in a separable set of operators $\cT\subset\cB(\cH)$,  we say $\cR$ is a F\o lner sequence for $\cT$.

(\emph{Notation:} When, in addition, $\cR:=\{R_n\}_{n\in\N}$ is exhausting, we will often use the symbol $\cP:=\{P_n\}_{n\in\N}$ as in Definition~\ref{def:qd}.)
\end{definition}

We summarize next some standard facts around these notions that will be revisited when we extend them for unbounded operators.

\begin{remark}\label{rem:QD/Folner_bd}
\begin{itemize}
 \item[a)] Every bounded block-diagonal operator (e.g., by the Peter-Weyl theorem, any unitary representation of a compact group) is clearly quasidiagonal. Moreover, \emph{any} sequence of projections associated with a filtration is quasidiagonalizing  for any compact operator. As a consequence, quasidiagonalizing sequences are stable under compact perturbations due to the linearity of the commutator. In particular, bounded normal operators, and therefore self-adjoint and unitary ones, are quasidiagonal (see \cite[Theorem~1]{berg:71} and \cite[Theorem~2]{Sikonia70} or \cite{bDavidson96,bChavan21}).

 \item[b)] The notion of a F\o lner sequence for an operator can be understood as a quasidiagonality condition relative to the growth of the dimension of the underlying spaces. In fact, any quasidiagonalizing sequence is also a F\o lner sequence, showing that quasidiagonality is a stronger notion. The unilateral shift $S$ on $\ell_2(\N)$ is a prototype that has a F\o lner sequence
 (take the orthogonal projection $P_n$ onto the subspace generated by the first $n$ canonical basis elements in $\ell_2(\N)$) but can not be quasidiagonal since it has Fredholm index $-1$. Finally, in the context of bounded operators, in Eq.~\eqref{eq:foelner} one can take, equivalently, the Hilbert-Schmidt norm $\|\cdot\|_2$ instead of the trace-class norm
 (see, e.g., \cite[Lemma~1]{Bedos97}).
 
\item[c)] Operator algebras come naturally into this analysis. For example, if $T\in\cB(\cH)$ is a quasidiagonal operator, then the C*-algebra  $C^*(T)$ generated by $T$ (i.e., the closure of all polynomials in $T$, $T^*$ and $\1$) is also quasidiagonal. Indeed, by Halmos' theorem, $T=B+K$, where $B$ is block-diagonal and $K$ is compact. Hence, if $B$ is block-diagonal with respect to $\{P_n\}_{n\in\N}$, then $C^*(T)\subset \{P_n\}_{n\in\N}'+\cK(\cH)$, so the sequence $\{P_n\}_{n\in\N}$ is quasidiagonalizing for every element in $C^*(T)$. (Here $\{P_n\}_{n\in\N}'$ denotes the operators in $\cB(\cH)$ commuting with the set of projections.) Furthermore,
given a sequence of projections $\cP$ associated with a filtration of $\cH$, the set of operators for which $\cP$ is a quasidiagonalizing/F\o lner sequence is a closed $*$-subalgebra of $\cB(\cH)$, i.e., a C*-algebra (see, for example, \cite[Chapter~7]{bBrown08}).

\item[d)] Since F\o lner sequences need not be exhaustive, we recall here that in Theorem~3.2 of \cite{LledoYakubovich13} a characterization of operators having a F\o lner sequence, but not an exhaustive one, is given. Such an operator $T\in\cB(\cH)$ can be represented as an orthogonal sum $T=F\oplus \widetilde{T}$ on $\cH=\cH_0\oplus \widetilde{\cH}$, where $F$ is a finite-dimensional operator and $\widetilde{T}\in\cB(\widetilde{\cH})$ has no F\o lner sequence in a strong sense
(see \cite[Definition~3.1]{LledoYakubovich13}). Moreover, quasinormal, subnormal, hyponormal and any essentially normal operators have increasing and exhausting F\o lner sequences. 
This characterization based on a finite-type reduction can be also stated in other areas of mathematics: in a pure algebraic setting see Theorem~3.9 in \cite{ALLW-1} or Section~4 in \cite{ALLW-2} in the context of C*-algebras.
(Note also that exhaustive sequences of finite-rank projections are called \emph{proper} in the references mentioned in this item since they are associated to proper filtrations of the underlying Hilbert space.)
\end{itemize}
\end{remark}
\subsection{Construction of quasidiagonalizing sequences for bounded normal operators}\label{subsec:berg}
As mentioned before, any bounded normal operator $N$ has a quasidiagonalizing sequence. 
There are several approaches to show this result based on a dyadic partition of a compact set containing the spectrum of the operator. We will not follow Berg's ingenious construction for 
normal operators (see \cite{berg:71}), but go along the route proposed in 
\cite[Section~2.2]{bChavan21}. We adapt their strategy to construct first an explicit quasidiagonalizing sequence for a self-adjoint operator and, then show that this sequence can be used as well for normal operators using a theorem due to Halmos \cite{Halmos72} stating that any 
normal operator is a continuous function of a self-adjoint operator.
One reason for following this strategy is that the first step suffices in many applications, in particular, in many examples appearing in mathematical physics where one has to approximate the self-adjoint Hamiltonian of a quantum mechanical system. In Section~\ref{subsec:berg-unbounded} we extend this construction to unbounded operators.\\

We divide the construction of the quasidiagonal sequences into two steps:\\

The first step is based on the following elementary observation: decomposing a
Hilbert space vector with respect to the spectral projections of a bounded 
self-adjoint operator associated with sufficiently small intervals yields a finite-dimensional subspace that is almost invariant under the operator. The precise statement is given in the following lemma. We denote by $E_A$ the spectral resolution of a self-adjoint operator $A$, and we write $P^\perp:=\mathbbm{1}-P$ for the projection onto the orthogonal complement of $P$.

\begin{lemma}[cf.~{\cite[proof of Lemma~2.2.1]{bChavan21}}]\label{Lemma: keystep_QD_sequence}
Let $A\in\cB(\cH)$ be a self-adjoint operator, $\varepsilon>0$ and
$\psi\in\cH$. Let $\{I_j\}_{j=1}^{k}$ be a Borel partition of a compact interval
$J\supseteq\sigma(A)$ into intervals of length at most $\varepsilon$, and let
$P$ be the orthogonal projection onto
\[
    \operatorname{span}\{E_A(I_j)\psi \mid 1\le j\le k\}\;.
\]
Then $\psi\in\operatorname{ran}P$, and $F:=P^\perp AP+PAP^\perp$ is a finite
rank self-adjoint operator such that $P$ reduces $A-F$ and
$\norm F\le\varepsilon/2$.
\end{lemma}

\noindent
 {\em Step~1:} \uline{Construction of a quasidiagonalizing sequence
$\cP$ for a self-adjoint operator $A\in\cB(\cH)$.} \\[1mm]
Let $\{\psi_n\}_{n\in\N}$ be an orthonormal basis of $\cH$ and fix $\varepsilon>0$. For each $n\in\N$, choose a partition $\{I_{n,j}\}_{j=1}^{k_n}$ of $[-\norm A,\norm A]$ into Borel intervals of length at most $\varepsilon/2^n$ (for instance, choosing the dyadic refinement at each step). We construct a sequence of projections $P_n$, $n\in\N$, recursively and set $P_0:=0$. Given $P_{n-1}$, write $\cH_n:=P_{n-1}^\perp\cH$ and let
$A_n:=P_{n-1}^\perp A|_{\cH_n}\in\cB(\cH_n)$ be the compression of $A$ to
$\cH_n$. Note that $A_n$ is self-adjoint and $\norm{A_n}\le\norm A$, so
$\sigma(A_n)\subset[-\norm A,\norm A]$. Let $Q_n$ denote the orthogonal projection onto
\[
    \operatorname{span}\{E_{A_n}(I_{n,j})P_{n-1}^\perp\psi_n \mid 1\le j\le k_n\}
    \subseteq\cH_n\;,
\]
and define $P_n:=\1_{P_{n-1}\cH}\oplus Q_n$ with respect to the decomposition
$\cH=P_{n-1}\cH\oplus\cH_n$.\\
We next show that $\cP:=\{P_n\}_{n\in\N}$ is quasidiagonalizing for $A$. Since $Q_n$ is the projection given by
Lemma~\ref{Lemma: keystep_QD_sequence} for the operator $A_n$, the vector
$P_{n-1}^\perp\psi_n$, and the partition $\{I_{n,j}\}_{j=1}^{k_n}$,  the finite-rank self-adjoint operator
$F_n:=Q_n^\perp A_nQ_n+Q_nA_nQ_n^\perp$ on $\cH_n$, where
$Q_n^\perp:=\1_{\cH_n}-Q_n$, satisfies $\norm{F_n}\le\varepsilon/2^{n+1}$. Extending each $F_n$ by zero to $\cH$ we have that the series 
$K:=\sum_{n\in\mathbb N}F_n$ converges in operator norm to a compact self-adjoint operator, and $ \norm{K}\le\varepsilon/2$  (note that the spaces $\cH_n$ are nested, not mutually orthogonal, so the sum is not an orthogonal sum; the bound follows from $\norm{K}\le\sum_n\norm{F_n}$).

Consequently, $B:=A-K$ is also self-adjoint. We claim that $[B,P_n]=0$ for every $n\in\N$. Since $B=B^*$, it suffices to show that
$P_n^\perp BP_n=0$, or equivalently, $P_n^\perp KP_n=P_n^\perp AP_n$. A direct calculation gives $F_n=P_n^\perp AQ_n+Q_nAP_n^\perp$ and therefore,
\[
    P_n^\perp KP_n=\sum_{i\le n}P_n^\perp AQ_i=P_n^\perp AP_n\;.
\]
Finally, $\cP$ is an increasing sequence of finite-rank
projections with $\psi_1,\ldots,\psi_n\in\operatorname{ran}P_n$ (by the Lemma,
$P_{n-1}^\perp\psi_n\in\operatorname{ran}Q_n$), and hence
$\lim_n P_n =\1$ strongly. By Remark~\ref{rem:QD/Folner_bd}\,(a), $\cP$
quasidiagonalizes the compact operator $K$. Since $[B,P_n]=0$, it follows
that $\cP$ also quasidiagonalizes $A$. In fact $\norm{[A,P_n]}=\norm{[K,P_n]}\le 2\norm K\le\varepsilon$ for all $n$.\\

\noindent
 {\em Step~2:} \uline{Given a normal operator $N\in \cB(\cH)$, construction of a self-adjoint operator $A$ with $N\in C^*(A)$.}  \\[1mm]
 Let $E_N$ be the spectral resolution of $N$. For each $n\in\N$, fix a Borel partition
$\{\Delta_{n,j}\}_{j=1}^{k_n}$ of $\sigma(N)$  such that each $\Delta_{n,j}$ has diameter at most
$\varepsilon/2^n$. Enumerate the resulting countable family of projections
$\{E_N(\Delta_{n,j}) : n\in\N,\,1\le j\le k_n\}$ as $\{E_k\}_{k\in\N}$. Since each $2E_k-\1$  is a self-adjoint unitary, the series $  A:=\sum_{k\in\N}3^{-k}(2E_k-\1)$ converges in norm to a self-adjoint operator. Moreover, $N\in C^*(A)$ by \cite[Lemma~2.2.3]{bChavan21} (see also \cite{Halmos72}).  It follows from Remark~\ref{rem:QD/Folner_bd}\,(c) that every quasidiagonalizing sequence for $A$ is also a quasidiagonalizing sequence for $N$. In particular, the sequence constructed in Step~1 has this property.

We now show that the sequence can be chosen so that $\norm{[N,P_n]}\le\varepsilon$ for every $n\in \N$.
 Since $N\in C^*(A)$, there exists a polynomial $p(x)=\sum_{i=0}^da_ix^i$ such that $\|N-p(A)\|\le\varepsilon/4$.
 For each $i\geq 1$, the commutator identity
\[[A^i,P_n]=\sum_{j=0}^{i-1} A^j[A,P_n]A^{i-1-j}\]
gives $\norm{[A^i,P_n]}\le i\:\norm{A}^{i-1} \norm{[A,P_n]}$. Since the constant term $a_0\1$ commutes with $P_n$, we obtain  $\norm{[p(A),P_n]}\leq c\:\norm{[A,P_n]}$ with $c=\sum_{i=1}^d i\,|a_i|\: \norm{A}^{i-1}$ (where $c=0$ if $p$ is constant, in which case $[p(A),P_n]=0$). Apply Step~1 to $A$ with $\widetilde \varepsilon:=\varepsilon/2(c+1)$ to obtain a quasidiagonalizing sequence $\{P_n\}_{n\in\N}$ satisfying
$$\|[A,P_n]\|\le \widetilde \varepsilon$$ for every $n\in \N$. Then, writing $[N,P_n]=[N-p(A),P_n]+[p(A),P_n]$, we obtain
\[\norm{[N,P_n]}\leq 2\norm{N-p(A)}+\norm{[p(A),P_n]}\le \frac{\varepsilon}{2}+c\:\widetilde{\varepsilon} 
\le\frac{\varepsilon}{2}+
\frac{c}{(c+1)}\;\frac{\varepsilon}{2}\leq \varepsilon,
\]
as required.

\begin{remark}
Using the notation of Step~1, note that
since $[B,P_n]=0$ for every $n\in\N$, the operator $B$ is block-diagonal with respect to $\{Q_n\}_{n\in\N}$. Indeed, $B=\sum_{n\in\N}Q_nAQ_n$ where the series converges in the strong operator topology (see, for example, \cite{Halmos70,bDavidson96}). As each finite-rank compression $Q_nAQ_n$ is self-adjoint, hence diagonalizable, this construction also provides a proof for Weyl's theorem  (see~\cite{Weyl09}). In \cite{berg:71}, Berg showed that every normal operator is also the compact perturbation of a diagonal operator. Since the eigenvectors of the diagonal operator are obtained within the proof, this approach allows to define an alternative quasidiagonalizing sequence for the normal operator by taking an exhaustive sequence of projections $P_n$ projecting onto the finite dimensional subspaces generated by the first $n$ eigenvectors. 
\end{remark}

\subsection{Sparse F\o lner sequences}\label{sec:sparse}
The aim of this section is to introduce sparse F\o lner sequences and demonstrate how this concept can be applied, in particular, to quasidiagonal operators. This allows to find sequences of projections relative to the choice of an orthonormal basis that not only satisfy the F\o lner condition but are also sparse, which might be useful in areas like numerical analysis, model order reduction and spectral approximation of operators (see, for example, \cite{bHagen01,Herrero82,benner15}). 

\smallskip

In this regard, it is worth remarking that in Fourier Analysis and Number Theory, sparse sequences (often called \emph{lacunary} when gaps grow exponentially) are integer sequences whose elements grow rapidly enough to be \emph{thin} in $\mathbb{Z}$.  
Indeed, in Fourier Analysis, lacunary sequences index Fourier series with remarkable properties: they often behave like sums of independent random variables (Salem–Zygmund Theorem) and converge almost everywhere despite large gaps. Likewise, in Number Theory, sparse sequences appear in additive combinatorics as thin sets avoiding arithmetic progressions and in studying additive bases (see \cite{Katznelson,Rudin,Nathanson} as basic references regarding their role).

To motivate the difference between quasidiagonal and F\o lner-type approximations we first consider the case of an exhaustive sequence of projections and a compact operator.  Recall from Remark~\ref{rem:QD/Folner_bd} that any sequence $\cP:=\{P_n\}_{n\in\N}$ associated with a filtration is quasidiagonalizing for any compact operator $K\in \cK(\cH)$. By the singular value decomposition, every compact operator $K$ can be written as $K=\sum_{j=1}^\infty u_j\otimes v_j$, for some $u_j,v_j\in \cH$. Hence, the sequence of finite-rank operators $\{K^{(N)}\}_{N\in \N}$ defined by $K^{(N)}=\sum_{j=1}^N u_j\otimes v_j$ converges to $K$ in operator norm, and we have
$$[K^{(N)},P_n]=\sum_{j=1}^N[u_j\otimes v_j,P_n]=\sum_{j=1}^N(u_j-P_nu_j)\otimes P_nv_j+P_nu_j\otimes(P_nv_j-v_j)\;.$$
Since $\cP$ is exhaustive, we have $\lim_n\norm{u_j-P_nu_j}=0$ and $\lim_n\norm{v_j-P_nv_j}= 0$  for all $j$, and thus,
$$\norm{[K^{(N)},P_n]}\leq\sum_{j=1}^N\norm{u_j-P_nu_j}\norm{P_nv_j}+\norm{P_nu_j}\norm{P_nv_j-v_j}
\overset{n\to \infty}{\longrightarrow}0\;,$$
which implies $\lim_n\norm{[K,P_n]}= 0$.

Now suppose instead that $\{k_j\}_{j\in\N}$ enumerates a proper
subset of $\N$, as is typically the case for a F\o lner type sequences. Then the
associated projections $R_n$ onto $\mathrm{span}\{e_{k_1},\dots,e_{k_n}\}$
need not converge to $\1$ in the strong
operator topology, and there is accordingly no reason to expect
$\norm{u_j-R_nu_j}\to0$ or $\norm{v_j-R_nv_j}\to0$ and, in general,
$\norm{[K,R_n]}$ need not tend to $0$ either. A F\o lner-type
approximation is nonetheless still possible, since the asymptotic
expression now averages the (possibly non-vanishing) commutator over the
growing rank of the projections.
In fact, we show next that this weaker form of convergence already holds under the sole condition of $\operatorname{rank}(R_n)\to\infty$, without any reference to the underlying basis. In fact, the following result shows that any rank-increasing 
sequence of projections is a F\o lner sequence for any compact operator.

\begin{example}\label{Ex: sparse_compact} 
Let $\{R_n\}_{n\in\N}$ be a sequence of non-zero, finite-rank orthogonal
projections on $\cH$ with $\operatorname{rank}(R_n)\to\infty$, and let
$K\in\cK(\cH)$. 
Then for every $\varepsilon > 0$, there exists a finite-rank operator $K_\varepsilon$ such that $\norm{K-K_\varepsilon}< \varepsilon/4$, and thus,
\[
\norm{[(K-K_\varepsilon), R_n]}_2  \le 2\norm{R_n}_2 \|K - K_\varepsilon\|\leq \varepsilon\norm{R_n}_2/2\;.
\]
Let now $N_\varepsilon\in \N$ be such that $\norm{R_n}_2\geq 4\norm{K_\varepsilon}_2/\varepsilon$ for every $n\geq N_\varepsilon$. In particular, this implies $\norm{R_n}_2\geq 2\norm{[K_\varepsilon,R_n]}_2/\varepsilon$. Hence,
 \[
\frac{\norm{[K, R_n]}_2}{\norm{R_n}_2} \le \frac{\norm{[K_\varepsilon, R_n]}_2}{\norm{R_n}_2} + \frac{\norm{[(K-K_\varepsilon), R_n]}_2}{\norm{R_n}_2}\leq \frac{\varepsilon}{2}+\frac{\varepsilon}{2}\;.
\]
\end{example}

\

{\bf Coordinatized projection subspaces:} In the following we need to describe more precisely the finite-rank projections associated to F\o lner sequences. Motivated by the result mentioned in Remark~\ref{rem:QD/Folner_bd}\,(d) we restrict here to the important case of increasing (not necessarily exhausting) sequences of projections.
For this fix an orthonormal basis $\cB=\{e_i\}_{i\in\N}$ of the Hilbert space $\cH$, and let $\cR:=\{R_n\}_{n\in\N}$ be an increasing sequence of non-zero, finite-rank orthogonal projections on $\cH$.

\

We will base our coordinate description of the projection subspaces on the choice of the following strictly increasing sequences of $\N$:
\begin{itemize}
    \item[i)] $\{k_j\}_{j\in\N}$ will record the indices (written in increasing order) of the basis elements of $\cB$ that will span the finite-dimensional subspaces associated to the projections $R_n$;
    \item[ii)] $\{r_n\}_{n\in\N}$ will correspond to the rank of the projections $R_n$, i.e. \ $\dim (R_n\cH)=r_n$.
\end{itemize}
In term of these choices we can coordinatize the projection subspace as follows:
\begin{equation}\label{eq:sparse}
 \operatorname{ran}(R_n)=\operatorname{span}\{e_{k_1},e_{k_2},\ldots,e_{k_{r_n}}\}\;,\qquad n\in\N\;.   
\end{equation}
Note that if $\{k_j\}_{j\in\N}$ is not exhaustive, then the corresponding sequence of projections will be increasing and not exhaustive as well and, hence, $\lim_{n}R_n$ will converge strongly to a proper subprojection of $\1$. 

The following definition formalizes the same underlying idea: 
the projections $R_n$ select a \emph{thin} subset of the basis $\cB=\{e_i\}_{i\in\N}$, where thinness is measured in terms of the asymptotic growth of the gaps that appear in the selection 
of basis elements generating the projection subspaces.

\begin{definition}\label{def:gapped-sparse}
Let $\cB=\{e_i\}_{i\in\N}$ be an orthonormal basis of $\cH$ and let $\cR:=\{R_n\}_{n\in\N}$ be a sequence of non-zero, finite-rank orthogonal projections on $\cH$ considered in Eq.~\eqref{eq:sparse}. The sequence of non-exhausting projections 
$\cR$ is called \emph{sparse with respect to $\cB$} if 
\[
\lim_{n\to \infty} \frac{r_n}{k_{r_n}} =0 \quadtext{(i.e. if} r_n=o (k_{r_n})\,)\;.
\]
If, in addition, $\cR$ is a F\o lner sequence for $T\in\cB(\cH)$, we call it a 
\emph{sparse F\o lner sequence} for $T$.
\end{definition}

To show the main result of this section we will coordinatize the projection subspaces in relation to the block diagonal operators given by Halmos theorem.
By Remark~\ref{rem:QD/Folner_bd}\,(c) (see also \cite[Section~4]{Halmos70}), every quasidiagonal operator
$T\in\cB(\cH)$ can be written as $T=B+K$, where $K$ is compact and $B$ is block-diagonal with respect to an exhaustive
sequence $\{Q_n\}_{n\in\N}$ of mutually orthogonal non-zero finite-rank projections. Throughout this section, we write
$$q_n:=\operatorname{rank}(Q_n)
\quadtext{and} s_n:=q_1+\cdots+q_n
\quadtext{with} (s_0:=0)\;.
$$
Moreover, we fix an orthonormal basis $\cB=\{e_i\}_{i\in\N}$ of $\cH$ adapted to $\{Q_n\}_{n\in\N}$, i.e.\ such that the set $\{e_i \mid s_{n-1}<i\le s_n\}$ is an orthonormal basis of $Q_n\cH$ for every $n\in\N$. 

\begin{theorem}\label{Thm: gapped_sparse_QD}
Let $T\in\cB(\cH)$ be a quasidiagonal operator, with $T=B+K$ as above. Then
for every strictly increasing and non-exhausting sequence $\{\ell_j\}_{j\in\N}$ in $\N$, the family
of orthogonal projections $\cR:=\{R_n\}_{n\in\N}$ defined by 
\begin{equation*}
    R_n:=Q_{\ell_1}+\cdots+Q_{\ell_n}\;,\qquad n\in\N\;.
\end{equation*}
is an increasing and non exhausting F\o lner sequence for $T$. 

If, in addition, 
\begin{equation}\label{Eq: sparsity_condition}
      \liminf_{n\to\infty} \frac{q_n}{s_{n-1}}=0       
\end{equation}
then $\{\ell_j\}_{j\in\N}$ can be chosen so that $\cR$ is a sparse F\o lner sequence for $T$ with respect to $\cB$.
\end{theorem}
\begin{proof}
Let $\{\ell_j\}_{j\in\N}$ be an arbitrary strictly increasing non-exhausting sequence in $\N$, then, by construction, we have that $\cR$ is a non-exhausting and increasing sequence of projections. To show it is a F\o lner sequence for $T$ recall that the projections 
$\{Q_m\}_{m\in\N}$, hence also  $\{R_n\}_n$, reduce the block diagonal operator $B$. Therefore
\[
    \frac{\|[T,R_n]\|_2}{\|R_n\|_2} = \frac{\|[K,R_n]\|_2}{\|R_n\|_2}
    \quadtext{and}
    \lim_{n\to\infty} \frac{\|[T,R_n]\|_2}{\|R_n\|_2}=0\;,
\]
where we use $\operatorname{rank}(R_n)\ge n\to\infty$ and Example~\ref{Ex: sparse_compact}. 
In conclusion $\cR$ is in fact a F\o lner sequence for $T$.

To show sparsity, define
\[
    I_j:=\{i\in\N \mid s_{\ell_j-1}<i\le s_{\ell_j}\}\;,\qquad j\in\N\;,
\]
where $s_{\ell_0}:=0$. Then, $\operatorname{ran}(Q_{\ell_j})=\operatorname{span}\{e_i\mid i\in I_j\}$ and $|I_j|=q_{\ell_j}$. Enumerate $I_1$ in increasing order as $\{k_1,\dots,k_{r_1}\}$, and, for each $n>1$, enumerate $I_n$ in increasing order as $\{k_{r_{n-1}+1},\dots,k_{r_n}\}$. Then, $\{k_j\}_{j\in \N}$ is strictly increasing and $k_{r_n}=s_{\ell_n}$. Moreover, the rank of the projection $R_n$ is
$r_n:=|I_1|+\cdots+|I_n|$, and hence
\[
    \operatorname{ran}(R_n)
    =\operatorname{span}\{e_i \mid i\in I_1\cup\cdots\cup I_n\}
    =\operatorname{span}\{e_{k_1},\ldots,e_{k_{r_n}}\}\;.
\]
By hypothesis in Eq.~\eqref{Eq: sparsity_condition} we can choose a strictly increasing sequence 
$\{\ell_n\}_{n\in\N}$ in $\N$, with $s_{(\ell_1-1)}>0$ and such that
\[
    \frac{q_{\ell_n}}{s_{(\ell_n-1)}} < \frac{1}{n}
    \quadtext{for every} n\in\N\;.
\]
Moreover, since $s_m\to\infty$ and, by Eq.~\eqref{Eq: sparsity_condition}, for each $n$ there are infinitely many indices $m$ with $q_m/s_{m-1}<1/n$, we may choose the $\ell_n$ inductively (enlarging the gaps between the blocks) so that in addition
$s_{\ell_{(n-1)}} \le \tfrac1n\,s_{(\ell_n-1)}$ for all $n\ge2$. Consequently, for every $n\ge 2$,
\[
    \frac{r_n}{k_{r_n}} 
    = \frac{r_n}{s_{\ell_n}}
    \le \frac{s_{\ell_{(n-1)}}}{s_{\ell_n}} + \frac{q_{\ell_n}}{s_{\ell_n}}
    \le \frac{s_{\ell_{(n-1)}}}{s_{(\ell_n-1)}} + \frac{q_{\ell_n}}{s_{(\ell_n-1)}}
    \le \frac{2}{n}
    \xrightarrow[n\to\infty]{} 0\;.
\]
Thus $\cR$ is a sparse F\o lner sequence for $T$ with respect to $\cB$.
\end{proof}

\begin{remark}
\begin{itemize}
    \item[a)] There are several important classes of operators satisfying the condition in 
    Eq.~\eqref{Eq: sparsity_condition}, hence allowing a sparse F\o lner sequence: 
    block-diagonal operators with uniformly bounded blocks, that is, $\sup_n q_n<\infty$ (see, e.g., \cite[Definition~16.1.5]{bBrown08}). In particular, by Berg's theorem \cite{berg:71}, every normal operator is a compact perturbation of a diagonal operator, so the corresponding reducing projections $Q_n$ have all rank one.  Thus self-adjoint, unitary and normal operators, and their compact perturbations, satisfy the hypothesis as well. By the Brown-Douglas-Fillmore theorem (see, e.g., \cite{bDavidson10}), the essentially normal quasidiagonal operators are compact perturbations of normal operators, so they are covered as well.
    
    \item[b)] If the condition in Eq.~\eqref{Eq: sparsity_condition} fails, then $q_n\ge c\:s_{n-1}$ for some $c>0$ and all sufficiently large $n$, so that $s_n\ge(1+c)s_{n-1}$. Iterating this inequality gives $s_n\ge(1+c)^{n-n_0}s_{n_0}$ for some sufficiently large $n_0$. Hence, the ranks $q_n$ grow at least geometrically.
\end{itemize}

\end{remark}


\section{Quasidiagonality and F\o lner sequences for unbounded operators}\label{sec:unbounded}
It is natural to ask about the quasidiagonal or F\o lner type approximations in the context of unbounded linear operators on Hilbert spaces. In fact, finite element methods typically ask for a concrete finite 
dimensional approximation of elliptic operators approximating weak solutions via sufficiently regular elements with the correct boundary conditions (see, for example, \cite[Section~3]{YPP17} and references therein). Also Hamiltonians in mathematical physics are unbounded in general and require different types of finite dimensional approximations (see \cite{Ibort13,FBL24}).

To give a unified presentation of the analysis, we will focus on the case of a densely defined closable operator $T\in \cL(\cH)$ with domain $\cD(T)$ and denote it by the pair $(T,\cD(T))\in\cL(\cH)$; we will also consider an orthonormal basis $\cB:=\{\psi_n\}_{n\in\N}\subset \cD(T)\subset\cH$ 
and extended seminorms given by
\begin{equation}\label{eq:norms}
   \norm{T}_u:=\sup_{\substack{\psi\in\cD(T)\\
                     \norm{\psi}\leq 1}}
                     \norm{T\psi},\quad
   \norm{T}_2:=\sqrt{{\sum_{n\in\N}\norm{T\psi_n}^2}} \quadtext{and} 
   \norm{T}_1:=\sum_{n\in\N}
                     \left\langle
                     \left|\overline{T}\right|\psi_n,\psi_n
                     \right\rangle\;.
\end{equation}
Note that the last term in Eq.~\eqref{eq:norms} is well-defined since 
$\cD(T)\subset\cD\left(\overline{T}\right)=\cD\left(\left|\overline{T}\right|\right)$ and that one 
can always find an orthonormal basis in $\cD(T)$ by applying the Gram-Schmidt algorithm to an infinite, linearly independent set of the domain (which is a dense linear subspace). The preceding notions are independent of any choice of basis within the domain.

\begin{definition}\label{Def:QDunbounded}
Let $(T,\cD(T))\in \cL(\cH)$ be a densely defined closable operator and $\cP=\{P_n\}_{n\in \N}$ be a sequence of projections associated with a filtration of $\cH$ satisfying the domain stability condition $P_n\cH\subseteq \cD(T)$, $n\in \N$.
We say that $\cP$ is a \textit{quasidiagonalizing sequence} for $T$ if $\;\lim\limits_{n} \norm{TP_n-P_nT}_u=0\,$; in this case we call $(T,\cD(T))$ a \textit{quasidiagonal operator}.
\end{definition}

\begin{definition}\label{Def:FOEunbounded}
Let $(T,\cD(T))\in \cL(\cH)$ be a densely defined closable operator and $\cR:=\{R_n\}_{n\in \N}$ be a sequence of non-zero finite-rank orthogonal projections satisfying $R_n\cH\subseteq \cD(T)\cap \cD(T^*)$, $n\in \N$. The sequence $\cR$ is called 
\begin{itemize}
\item[(i)]  a \textit{1-F\o lner sequence} for $T$ if 
$\quad\lim\limits_{n} \frac{\norm{TR_n-R_nT}_1}{\norm{R_n}_1}=0\;$,

\item[(ii)] a \textit{2-F\o lner sequence} for $T$ if 
$\quad\lim\limits_{n} \frac{\norm{TR_n-R_nT}_2}{\norm{R_n}_2}=0\;$. 
\end{itemize}
\end{definition}

Note that if the operator $T$ is symmetric then the domain stability in the latter definition reduces simply to $R_n\cH\subseteq \cD(T)$.

In the next propositions, we collect some consequences of the preceding definitions and present some relations between the previous notions in the context of unbounded operators. In particular, compare the following results with Lemma~1 and Propostion~2 in \cite{Bedos97}.

\begin{proposition}
    Let $(T,\cD(T))\in \cL(\cH)$ be a densely defined closable operator and let $R$ be a finite-rank orthogonal projection of rank $r$ satisfying $R\cH\subseteq\cD(T)\cap\cD(T^*)$. Then, $[T,R]$ extends to a bounded operator on $\cH$ of rank at most $r$ and
\begin{equation}\label{eq:bound-commutator}
\frac{\norm{[T,R]}_1}{\norm{R}_1}\leq \sqrt{2} \;\;\frac{\norm{[T,R]}_2}{\norm{R}_2}
\leq 2\;\norm{[T,R]}_{u}\;.
\end{equation}
In particular:
\begin{itemize}
\item[(1)] if $\cP=\{P_n\}_{n\in \N}$ is a quasidiagonalizing sequence for $T$, then $P_n\cH\subset\cD(T^*)$, and $\cP$ is a $2$-F\o lner sequence for $T$.
\item[(2)] Every $2$-F\o lner sequence for $T$ is a $1$-F\o lner sequence for $T$.
\end{itemize}
\end{proposition}
\begin{proof}

Let $\{\varphi_j\}_{j=1}^r$ be an orthonormal basis of $R\cH$. For every $\psi\in\cD(T)$, we have
\[
[T,R]\psi=TR\psi-\sum_{j=1}^r\langle\psi,T^*\varphi_j\rangle\varphi_j.
\]
The operator $TR$ is defined on all of $\cH$ and has rank at most $r$, since $R\cH\subseteq\cD(T)$.  Moreover, the second term is bounded on $\cD(T)$ and has rank at most $r$. 
Thus, $[T,R]$ extends to a bounded operator on $\cH$ of rank at most $2r$ and the following estimates hold 
$$\norm{[T,R]}_2\leq\sqrt{2r}\;\norm{[T,R]}_u\;,\quadtext{and}\norm{[T,R]}_1\leq\sqrt{2r}\;\norm{[T,R]}_2\;.$$
Consequently, using $\norm{R}_2=\sqrt r$ and $\norm{R}_1=r$ we have
$$\frac{\norm{[T,R]}_2}{\norm{R}_2}\leq \sqrt{2}\,\norm{[T,R]}_{u}\quadtext{and}
\frac{\norm{[T,R]}_1}{\norm{R}_1}\leq \frac{\sqrt{2}\norm{[T,R]}_2}{\norm{R}_2}\;.$$

To show (1) assume now that $\{P_n\}_{n\in \N}$ is a quasidiagonalizing sequence for $T$. For every $\varphi\in P_n\cH$, we have
    \[
\langle T\psi,\varphi\rangle
=\langle P_nT\psi,\varphi\rangle
=\langle TP_n\psi-[T,P_n]\psi,\varphi\rangle,\quad \psi\in \cD(T).
\]
Since $P_n\cH\subseteq\cD(T)$ is finite-dimensional, $TP_n$ is bounded. Moreover $[T,P_n]$ is bounded on $\cD(T)$ by the quasidiagonality assumption in Definition~\ref{Def:QDunbounded}. Hence, there exists a constant $C>0$ such that $|\langle T\psi,\varphi\rangle|\le C\Vert\psi\Vert$. It follows that $\varphi\in \cD(T^*)$.  Thus, $P_n\cH\subseteq\cD(T^*)$ for every $n\in\N$, and the second inequality of Eq.~\eqref{eq:bound-commutator} gives the desired result.

Finally, part (2)  follows from Eq.~\eqref{eq:bound-commutator}.
\end{proof}

\begin{proposition}\label{prop:stability}
 Let $(T,\cD(T))$ be an unbounded closable operator on $\cH$, $\cP=\{P_n\}_{n\in \N}$ a sequence of projections associated with a filtration of $\cH$, $\cR=\{R_n\}_{n\in \N}$ a sequence of non-zero finite-rank orthogonal projections and $a\in\{1,2\}$.
 
\begin{itemize}
\item[(1)] If $\cP$ 
is a quasidiagonalizing sequence for $T$, then it is a quasidiagonalizing sequence for any compact perturbation of $T$ with the same domain.
\item[(2)] If $\cR$ is an $a$-F\o lner sequence for $T$, then it is an
$a$-F\o lner sequence for its closure $\overline{T}$.
\end{itemize}
\end{proposition}
\begin{proof}
To show (i), note first that the stability of the domain under the finite-rank projections is immediate since $\cD(T+K)=\cD(T)$. The operator $T+K$ is closable, as it admits the closed extension $(\overline{T}+K,\cD(\overline{T}))$. Moreover, for any compact operator
$K$ on $\cH$, we have  on $\cD(T)$ 
 that $[(T+K),P_n]=[T,P_n]+[K,P_n]$ hence
\[
 \|[(T+K),P_n]\|_u \leq \|[T,P_n]\|_u+\|[K,P_n]\|\;.
\]
Therefore, $\lim_{n\to\infty} \|[(T+K),P_n]\|_u=0$ by Remark~\ref{rem:QD/Folner_bd}\,(a).\\Part (ii) follows directly from the definition of F\o lner sequence for unbounded operators and the fact that $\cD(T)\subset\cD (\overline{T})$ as well as $\cD\big(\,\overline{T}^*\big)=\cD(T^*)$.
\end{proof}
We now extend Halmos' theorem on the characterization of quasidiagonal operators as compact perturbations of block-diagonal operators to the unbounded scenario (cf.~\cite{Halmos70}). For this we need to extend the notion of block-diagonal operator:

\begin{definition}\label{Def:BLOCKunbounded}
Let $(B,\cD(B))\in \cL(\cH)$ be a densely defined closable operator and let $\cQ:=\{Q_n\}_{n\in \N}$ be a sequence of mutually orthogonal non-zero finite-rank projections of $\cH$  with $\sum_{n\in\N}Q_n=\1$ (convergence in the strong operator topology)
such that $Q_n\cH\subseteq \cD(B)$ for all $n\in \N$. We say that $\cQ$ is a 
\textit{block-diagonalizing sequence} for $B$ if for all $n\in\N$, $BQ_n-Q_nB=0$ on $\cD(B)$. 
In this case we say $(B,\cD(B))$ is a \textit{block-diagonal} operator.
\end{definition}

\begin{theorem}\label{Theorem:QD=BD+K}
    Let $(T,\cD(T))\in \cL(\cH)$ be a (closable) quasidiagonal operator. Then, for every $\varepsilon>0$, there exists a compact operator $K\in\cK(\cH)$ with $\norm{K}<\varepsilon$ such that $B:=T-K$ with $\cD(B)=\cD(T)$ is block-diagonal.
\end{theorem}
\begin{proof}
    Fix $\varepsilon>0$ and let $\cP=\{P_n\}_{n\in \N}$ be a quasidiagonalizing sequence for $T$. By dropping to a subsequence if necessary, we may assume that 
    $$\norm{TP_n-P_nT}_u<\varepsilon/2^{n+1} \quadtext{for all}n\in \N\;.$$ 
    Consider the sequence of mutually orthogonal finite-rank projections $Q_n:=P_n-P_{n-1}$, $n\in\N$
    (putting $P_0=0$) as in Section~\ref{sec:bounded}. Define the operator
    $$K=\sum_{n\in\N}Q_{n+1} T P_n+P_n T Q_{n+1}
       =\sum_{n\in\N}P_{n+1}(TP_n-P_nT)P_n-P_n(TP_n-P_nT)P_{n+1}\;,$$
    which is indeed compact as the series converges in norm by the estimate above, and since 
    $$\norm{P_{n+1}(TP_n-P_nT)P_n-P_n(TP_n-P_nT)P_{n+1}}\leq 2\norm{TP_n-P_nT}_u\;,$$
    each partial sum has finite-rank.
    Moreover, we have that $\norm{K}<\varepsilon$. 
    
    Now let $B:=T-K$, with $\cD(B)=\cD(T)$. This operator is closable, since  
    $(\overline{T}-K,\cD(\overline{T}))$ is a closed extension. We next show that $B$ commutes with the projections $P_m$. To this end, observe first that, for every $m\in\N$, the operator $P_mT=TP_m-[T,P_m]$ is bounded on $\cD(T)$ (since $TP_m$ has finite rank and $\norm{[T,P_m]}_u<\infty$); we denote its bounded extension to $\cH$ by the same symbol. Using the telescoping property of the sequence $\{Q_n\}_{n\in\N}$ we obtain, for every $m\in\N$,
    \begin{align*}
     P_m^\perp K P_m&=\sum_{n\ge m}Q_{n+1}TP_m=P_m^\perp TP_m\;,\\
     P_m K P_m^\perp\psi&=\sum_{n\ge m}P_mTQ_{n+1}\psi=P_mTP_m^\perp\psi\;,\qquad \psi\in\cD(T)\;,
    \end{align*}
    where in the second identity we used that $\sum_{n\ge m}Q_{n+1}\psi=P_m^\perp\psi\in\cD(T)$ converges in norm and that $P_mT$ is bounded. Hence $P_m^\perp(T-K)P_m=0$ and $P_m(T-K)P_m^\perp=0$ on $\cD(T)$, so that $[(T-K),P_m]=0$ and therefore $[(T-K),Q_n]=0$ on $\cD(T)$ for all $n\in\N$. In particular, since $Q_nKQ_n=0$ for all $n\in\N$, we obtain
    $$B\psi=(T-K)\psi=\sum_{n\in\N}Q_n T Q_n\psi\;,\qquad \psi\in\cD(T)\;,$$
    with convergence in norm.
   \end{proof}
Conversely, it is clear that every block-diagonal operator is quasidiagonal, and by Proposition~\ref{prop:stability}, so is every compact perturbation of a block-diagonal operator. As a consequence of Theorem~\ref{Theorem:QD=BD+K}, we obtain the following corollary, whose proof is analogous to that of Theorem~\ref{Thm: gapped_sparse_QD}.
\begin{corollary}\label{cor:usparse}
Let $(T,\cD(T)) \in \cL(\mathcal{H})$ be a quasidiagonal operator. Then $T$ admits an increasing sequence that is both $1$- and $2$-F\o lner. In addition, if the sequence $\{Q_n\}_{n\in\N}$ associated with the block-diagonal operator in Theorem~\ref{Theorem:QD=BD+K} satisfies Eq.~\eqref{Eq: sparsity_condition}, then the increasing sequence can be chosen to be sparse.
\end{corollary}
\subsection{Quasidiagonalizing sequences for unbounded normal operators}\label{subsec:berg-unbounded}

In this subsection, we generalize the construction in Section~\ref{subsec:berg} to unbounded self-adjoint and normal operators.

Let $(A,\cD(A))$ be an unbounded self-adjoint operator on $\cH$ with spectral resolution $E_A$, and let $\{I_k\}_{k\in\N}$ be a Borel partition of $\R$ into bounded intervals. Set $\cH_k:=E_A(I_k)\cH$. Then $\cH=\bigoplus_{k\in\N}\cH_k$. By the spectral theorem, each spectral projection $E_A(I_k)$ commutes with $A$, i.e. $AE_A(I_k)\psi=E_A(I_k)A\psi$ for every $\psi\in \cD(A)$, and hence each
$\cH_k$ reduces $A$. Moreover, since $I_k$ is bounded, $\cH_k\subset\cD(A)$ and the restriction $A_k:=A|_{\cH_k}$ is a bounded self-adjoint operator on $\cH_k$.
(If some $\cH_k$ is finite-dimensional, Step~1 below trivially yields a sequence that is eventually equal to $\1_{\cH_k}$, and if $\cH_k=\{0\}$ all its terms vanish; discarding the indices $n$ for which $P_n=0$ we may assume that all $P_n$ defined below are non-zero.)

For each $k\in\N$ apply \textit{Step~1} of Section~\ref{subsec:berg} to $A_k$ 
on $\cH_k$, with $\varepsilon_k:=2^{-k}$, to obtain a quasidiagonalizing sequence $\{Q^{(k)}_m\}_{m\in \N}$ for $A_k$ and such that $ \|[A_k,Q^{(k)}_m]\|\le2^{-k}$ for all $m\in\N$. For each $n\in\N$,  define $P_n:=\sum_{k\leq n}Q^{(k)}_{n+1-k}\;$ where each $Q^{(k)}_m$ is regarded as an operator on $\cH$ by extending it by zero to $\cH_k^\perp$.

\begin{proposition}
The sequence $\cP=\{P_n\}_{n\in\N}$ is a quasidiagonalizing sequence for
$(A,\cD(A))$.
\end{proposition}
\begin{proof}
Since the summands of $P_n$ have mutually orthogonal ranges, each contained in a different  $\cH_k$, the operator $P_n$ is a finite-rank orthogonal projection. Moreover, $\cP$ is increasing because $Q^{(k)}_{m}\leq Q^{(k)}_{m+1}$ for every $k,m\in \N$, and thus we have $P_{n+1}-P_n=\sum_{k\le n}\big(Q^{(k)}_{n+2-k}-Q^{(k)}_{n+1-k}\big)
+Q^{(n+1)}_1\ge 0$. The strong convergence $P_n\to \1$ follows from the strong convergence of each sequence $\{Q^{(k)}_m\}_{m\in \N}$ and the orthogonal decomposition $\cH=\bigoplus_{k\in\N}\cH_k$. Indeed, for 
 $\psi=\sum_k \psi_k$, with $\psi_k\in \cH_k$,
\[
 \|\psi-P_n\psi\|^2=\sum_{k\le n}\|\psi_k-Q^{(k)}_{n+1-k}\psi_k\|^2+\sum_{k>n}\|\psi_k\|^2
 \xrightarrow[n\to\infty]{}0.
\]
The domain stability condition is immediate, since $P_n\cH\subset\bigoplus_{k\le n}\cH_k\subset\cD(A)$. It remains to prove that $\lim_{n\to\infty}\|[A,P_n]\|_u=0$. Let $\psi\in\cD(A)$ and write $\psi=\sum_k\psi_k$, where $\psi_k=E_A(I_k)\psi\in\cH_k$. Since each $\cH_k$ reduces $A$ and $P_n$ acts on $\cH_k$ as $Q^{(k)}_{n+1-k}$ for
$k\le n$ and as zero for $k>n$, we obtain
 \[[A,P_n]\psi=\sum_{k\le n}[A_k,Q^{(k)}_{n+1-k}]\psi_k.\]
As the summands on the right belong to mutually
orthogonal subspaces,
 \[\Vert[A,P_n]\psi\Vert^2
 =\sum_{k\le n}\Vert[A_k,Q^{(k)}_{n+1-k}]\psi_k\Vert^2
  \leq\left(\max_{k\le n}\Vert[A_k,Q^{(k)}_{n+1-k}]\Vert\right)^2\,\Vert\psi\Vert^2\]
 and therefore
\[
 \|[A,P_n]\|_u\le\max_{k\le n}\Vert[A_k,Q^{(k)}_{n+1-k}]\Vert.
\]
Now, fix $\varepsilon>0$ and choose $k_0\in \N$ such that $2^{-k_0}<\varepsilon$. For the indices $k\geq k_0$, we have $\Vert[A_k,Q^{(k)}_{n+1-k}]\Vert\leq 2^{-k}<\varepsilon$ whenever $n\ge k$. For the remaining finitely many indices $k<k_0$, since $\{Q^{(k)}_m\}_{m\in\N}$ is a quasidiagonalizing sequence for $A_k$, there exists $n_0\in \N$ such that $\|[A_k,Q^{(k)}_{n+1-k}]\|<\varepsilon$ for every $k<k_0$ whenever $n\ge n_0$. Combining the two estimates we obtain $\|[A,P_n]\|_u<\varepsilon$ for all sufficiently large $n$. This completes the proof.
\end{proof}
Similarly, a quasidiagonalizing sequence for an unbounded normal operator $(N,\cD(N))$ is obtained by replacing the intervals $I_n$ by a partition of $\C$ into bounded Borel sets, and \textit{Step~1} by \textit{Step~2} of Section~\ref{subsec:berg}.

As a consequence of Proposition~\ref{prop:stability}\,(1) we obtain the following result.
\begin{corollary}\label{cor:qd-unbounded}
 The quasidiagonalizing sequence for the self-adjoint operator $(A,\cD(A))$  constructed in this subsection is stable under compact perturbations of $A$.
\end{corollary}

\subsection{Classes of examples}
In this section we illustrate how the typical algebraic and analytical limitation 
when working with unbounded operators also appear when one considers approximations
related to filtrations of the Hilbert space.
In contrast to the bounded operator setting, we will show that
\begin{itemize}
 \item the notions of $1$-F\o lner and $2$-F\o lner sequences for a closable operator are no longer equivalent;
 \item a quasidiagonalizing or F\o lner sequence for a closable operator need not have this property for higher-order monomials of this operator.
\end{itemize}
\medskip

We first introduce the notion of the propagation of an operator 
(see, e.g., \cite[Definition~5.9.2]{nowak-yu-12}).
Let $(T,\cD(T))\in\cL(\cH)$ and let 
$\cB:=\{\psi_n\}_{n\in\N}\subset\cD(T)\cap\cD(T^*)$ be an orthonormal basis of $\cH$. Identifying $\cH\cong\ell_2(\N)$, we write $T\cong(T_{ij})_{i,j\in\N}$ with $T_{ij}:=\langle T\psi_j,\psi_i\rangle$. The \textit{propagation} of $T$ with respect to $\cB$ is defined by
 \[
    \fp_{\cB}(T):=\sup\{ |i-j|\mid T_{ij}\not=0\}    \;.   
 \]
 Let $P_n$ be the orthogonal projection onto $\operatorname{span}\{\psi_1,\dots,\psi_n\}$ and assume
\begin{equation}\label{eq:finite_prop}
 \fp_{\cB}([T,P_n])<\infty\quad\text{for all }n\in\N .
\end{equation}
One can readily check that \eqref{eq:finite_prop} holds if and only if every row and every column of $(T_{ij})$ has finite support. Note that $\fp_\cB(T)$ itself may be infinite. Since 
$P_n\cH\subset\operatorname{span}\cB\subset\cD(T)\cap\cD(T^*)$, the domain stability condition of Definitions~\ref{Def:QDunbounded} and \ref{Def:FOEunbounded} is satisfied, and the extended seminorms of Eq.~\eqref{eq:norms} are finite for $[T,P_n]$.

The finite propagation condition above allow us to compute the extended seminorms of the commutators as the usual Schatten norms of their finite-dimensional corners.
\begin{proposition} \label{Prop: propagation} Let $(T,\cD(T),\{P_n\}_{n\in\N})$ be as above, satisfying \eqref{eq:finite_prop}.  For $m:=n+\fp_\cB([T,P_n])$,  we have $[T,P_n]=P_m[T,P_n]P_m$ on $\cD(T)$, and
$$\norm{[T,P_n]}_u=\norm{P_m[T,P_n]P_m}\quadtext{and}\norm{[T,P_n]}_a=\norm{P_m[T,P_n]P_m}_a
\;,\;\;a=1,2\,.$$
\end{proposition}
\begin{proof}
Write  $T\cong (T_{ij})_{i,j\in \N}$ and  $[T,P_n]\cong ([T,P_n]_{ij})_{i,j\in \N}$ for the matrices with respect to $\cB$. Then,
$$[T,P_n]_{ij}=\begin{cases}
    -T_{ij}&\text{if }i\leq n<j \text{ and } |i-j|\leq\fp_\cB([T,P_n]),\\[1mm]
    T_{ij}&\text{if }j\leq n<i \text{ and } |i-j|\leq\fp_\cB([T,P_n]),\\[1mm]
    0&\text{otherwise}.
\end{cases}$$
If $[T,P_n]_{ij}\ne0$, then $\max\{i,j\}=\min\{i,j\}+|i-j|\le n +\fp_\cB([T,P_n])=m$.  Thus all non-zero entries lie in the first $m$ rows and columns. Since $P_m\psi\in\cD(T)$ for every $\psi\in\cD(T)$, this gives $[T,P_n]\psi=P_m[T,P_n]P_m\psi$ on $\cD(T)$. The claim follows because $P_m[T,P_n]P_m$ is a bounded operator  and $\cD(T)$ is dense.
\end{proof}

\begin{remark}
Recall from Remark~\ref{rem:QD/Folner_bd}\,(b) that in the context of bounded 
operators the notions of $1$- and $2$-F\o lner sequences are equivalent. The proof
that $2$-F\o lner is also $1$-F\o lner sequence follows in spirit of proof of the
last implication in the preceding proposition. Nevertheless, to show the reverse implication one needs to use Powers-St\o rmer inequality 
(\cite[Lemma~4.1]{powers-70}) which does not hold in the context of unbounded operators. 
\end{remark}

\noindent  Throughout the rest of this subsection, we will denote by $\cB:=\{e_n\}_{n\in \N}$ the canonical basis of $\ell^2(\N)$ and by $\cP_\cB:=\{P_n\}_{n\in\N}$  the sequence of orthogonal projections $P_n$ onto $\operatorname{span}\{e_1,\dots,e_n\}$. For $i,j\in \N$, let $E_{i,j}$ denote the matrix unit associated with $\cB$, defined by $E_{i,j}e_k=\delta_{jk}e_i$, $k\in \N$.\\

We start by studying examples of operators with bounded propagation.
\
\begin{example}[Weighted shifts]\label{Ex: S_w}
Let $(S_w,\cD(S_w))$ be the unilateral weighted shift on $\ell_2(\N)$ with weight sequence $w=(w_n)_{n\in\N}$, i.e.\ $S_we_n=w_ne_{n+1}$, with domain
\[
\cD(S_w)=c_{00}:=\Big\{x=\sum_{n\in\N}\alpha_ne_n\;\Big|\;\alpha_n\neq0\text{ for finitely many }n\Big\}\subset\ell_2(\N)\;.\]
Since $c_{00}\subseteq\cD(S_w^*)$, the adjoint is densely defined and $S_w$ is closable. Its matrix has the entries $w_n$ on the first subdiagonal and zeros elsewhere.
 On $c_{00}$, we have
$[S_w,P_n]=w_n E_{n+1,n}$, so $\fp_\cB(S_w)=\fp_\cB([S_w,P_n])=1$. By Proposition~\ref{Prop: propagation} with $m=n+1$, 
\[
 \norm{[S_w,P_n]}_u=\norm{[S_w,P_n]}_a=|w_n|,\qquad\text{hence}\qquad
 \frac{\norm{[S_w,P_n]}_a}{\norm{P_n}_a}=\frac{|w_n|}{n^{1/a}},\quad a=1,2 .
\]
Therefore, depending on the weights, $\cP_\cB$ may or may not be quasidiagonalizing or $a$-F\o lner for $S_w$:

\begin{center}
\renewcommand{\arraystretch}{1.4}
\begin{tabular}{c|c|c|c}
$w_n$ & $\lim_n\norm{[S_w,P_n]}_u$ & $\lim_n\dfrac{\norm{[S_w,P_n]}_1}{\norm{P_n}_1}$ & $\lim_n\dfrac{\norm{[S_w,P_n]}_2}{\norm{P_n}_2}$\\[2mm]
\hline
$\log n$ & $\infty$ & $0$ & $0$\\
$n^{1/4}$ & $\infty$ & $0$ & $0$\\
$\sqrt{n}$ & $\infty$ & $0$ & $1$\\
$n$ & $\infty$ & $1$ & $\infty$
\end{tabular}
\end{center}
In particular:
\begin{itemize}
\item for $w_n=\log n$ and $w_n=n^{1/4}$, $\cP_\cB$ is not quasidiagonalizing for $S_w$, but it is an $a$-F\o lner for $a=1,2$;
\item for $w_n=\sqrt{n}$, $\cP_\cB$ is neither quasidiagonalizing nor $2$-F\o lner, but it is $1$-F\o lner;
\item for $w_n=n$, $\cP_\cB$ is neither quasidiagonalizing nor $a$-F\o lner for $a=1,2$.
\end{itemize}
\end{example}

\begin{remark}\label{Remark: differences_bd}\begin{itemize}
\item[a)] As noted in Remark~\ref{rem:QD/Folner_bd}\,(b), for bounded operators the F\o lner  condition can equivalently be formulated with $\norm\cdot_2$ (cf. \cite[Lemma~1]{Bedos97}). Example~\ref{Ex: S_w} (with $w_n=\sqrt n$) shows that this fails for unbounded operators: a $1$-F\o lner sequence need not be a $2$-F\o lner sequence.
\item[b)] If $T\in\cB(\cH)$ is bounded with $\fp_\cB(T)<\infty$, then $\cP_\cB$ is always a F\o lner sequence. Indeed, the commutator $[T,P_n]$ has at most $2\fp_\cB(T)$ non-zero columns, so the norms $\norm{[T,P_n]}_2$ are uniformly bounded:
\[\norm{[T,P_n]}_2=\sqrt{\sum_{k\in N}\norm{[T,P_n]e_k}^2}\le\sqrt{2\fp_\cB(T)}\,\norm{[T,P_n]}\le2\sqrt{2\fp_\cB(T)}\,\norm{T}\;.\]
 In contrast, for an unbounded operator $(T,\cD(T))$ with $\fp_\cB(T)<\infty$ the seminorms $\norm{[T,P_n]}_a$ need not be bounded and $\cP_\cB$ may fail to be a F\o lner sequence (Example~\ref{Ex: S_w} with $w_n=n$).
    \end{itemize}
\end{remark}

\begin{example}\label{Ex: S^2_w}
Let $(S^2_w,\cD(S^2_w))$ be the square of the operator of Example~\ref{Ex: S_w}, with domain $\cD(S_w^2)=c_{00}$. It is closable because $c_{00}\subseteq\cD((S_w^2)^*)$, and $S_w^2e_n=w_nw_{n+1}e_{n+2}$.  On $c_{00}$ (with the convention $w_0:=0$), $[S_w^2,P_n]=w_{n-1}w_n E_{n+1,n-1}+w_nw_{n+1}E_{n+2,n}$  and, therefore,
$\fp_{\cB}(S^2_w)=\fp_{\cB}([S^2_w,P_n])=2$ for all $n\in \N$.
By Proposition~\ref{Prop: propagation} with $m=n+2$,
\[
 \norm{[S_w^2,P_n]}_u=w_nw_{n+1},\quadtext{and}
 \frac{\norm{[S_w^2,P_n]}_a}{\norm{P_n}_a}=\left(\frac{w_{n-1}^aw_n^a+w_n^aw_{n+1}^a}{n}\right)^{1/a}\quad a=1,2\;.\]
For the weights $w_n=\sqrt{n}$ and $w_n=n^{1/4}$ we obtain:
\begin{center}
\renewcommand{\arraystretch}{1.4}
\begin{tabular}{c|c|c|c}
$w_n$ & $\lim_n\norm{[S_w^2,P_n]}_u$ & $\lim_n\dfrac{\norm{[S_w^2,P_n]}_1}{\norm{P_n}_1}$ & $\lim_n\dfrac{\norm{[S_w^2,P_n]}_2}{\norm{P_n}_2}$\\[2mm]
\hline
$n^{1/4}$ & $\infty$ & $0$ & $\sqrt{2}$\\
$\sqrt{n}$ & $\infty$ & $2$ & $\infty$
\end{tabular}
\end{center}
Therefore, for $w_n=n^{1/4}$, the sequence $\cP_\cB$ is $1$-F\o lner but neither $2$-F\o lner nor quasidiagonalizing for $S_w^2$. For $w_n=\sqrt{n}$, however, it is neither  $1$-F\o lner nor $2$-F\o lner, and it is not quasidiagonalizing for $S_w^2$.
\end{example}
\begin{remark}\label{rem:no_algebra}
In contrast to Remark~\ref{rem:QD/Folner_bd}\,(c), Examples~\ref{Ex: S_w} and \ref{Ex: S^2_w} show that $a$-F\o lner sequences are not stable under products. For $w_n=\sqrt n$ the sequence $\cP_\cB$ is $1$-F\o lner for $S_w$ but not for $S_w^2$, and for $w_n=n^{1/4}$ it is $2$-F\o lner for $S_w$ but not for $S_w^2$. Thus the set of (possibly unbounded) operators sharing a given $a$-F\o lner sequence is a linear space, but in general not an algebra.

The difference is that, for unbounded operators, the product rule $[ST,P_n]=S[T,P_n]+[S,P_n]T$ gives only
\[
 \norm{[ST,P_n]}_a\le\norm S_u\,\norm{[T,P_n]}_a+\norm{[S,P_n]}_a\,\norm T_u ,
\]
and the right-hand side is infinite when $\norm S_u=\norm T_u=\infty$, even if the commutators are arbitrarily small. The same happens for quasidiagonalizing sequences when one tries to apply the submultiplicativity of $\norm\cdot_u$, as the next example shows.
\end{remark}

\begin{example} Let $(A,\cD(A))$ be an operator on $\ell^2(\N)$ given by $Ae_{2j-1}=(2j-1)^2e_{2j-1}+1/(2j-1)e_{2j}$ and $Ae_{2j}=(2j)^2e_{2j}$
for $j\geq1$, with domain $\cD(A)=c_{00}$. Let also $(A^2,\cD(A^2))$, with domain $\cD(A^2)=c_{00}$. Both operators are block diagonal with $2\times2$ blocks on $\operatorname{span}\{e_{2j-1},e_{2j}\}$:

\[
A = \begin{pmatrix}
1^2 & 0 & 0 & 0 & \dots \\
1 & 2^2 & 0 & 0 & \dots \\
0 & 0 & 3^2 & 0 & \dots \\
0 & 0 & \frac13 & 4^2 & \dots \\
\vdots & \vdots & \vdots & \vdots & \ddots
\end{pmatrix},\qquad
A^2 = \begin{pmatrix}
1^4 & 0 & 0 & 0 & \dots \\
1^2+2^2 & 2^4 & 0 & 0 & \dots \\
0 & 0 & 3^4 & 0 & \dots \\
0 & 0 & \frac{3^2+4^2}{3} & 4^4 & \dots \\
\vdots & \vdots & \vdots & \vdots & \ddots
\end{pmatrix}.
\]
Both $A$ and $A^2$ are closable, since $c_{00}\subset\cD(A^*)$ and $c_{00}\subset\cD((A^2)^*)$. We claim that $\cP_\cB$ is a quasidiagonalizing sequence for $A$ but not for $A^2$. Indeed, $P_{2j}$ reduces both operators, so $[A,P_{2j}]=[A^2,P_{2j}]=0$, while on $c_{00}$,
\[
 [A,P_{2j-1}]=\tfrac{1}{2j-1} E_{2j,2j-1},\qquad
 [A^2,P_{2j-1}]=\tfrac{(2j-1)^2+(2j)^2}{2j-1}E_{2j,2j-1}.
\]
Therefore
\[
 \norm{[A,P_{2j-1}]}_u=\frac1{2j-1}\xrightarrow[j\to\infty]{}0,\qquad
 \norm{[A^2,P_{2j-1}]}_u=\frac{(2j-1)^2+(2j)^2}{2j-1}\xrightarrow[j\to\infty]{}\infty .
\]
\end{example}
We conclude with an example of an operator with unbounded propagation.
\begin{example}
Let $(S_+,\cD(S_+))$ be the operator on $\ell_2(\N)$ with domain $\cD(S_+)=c_{00}$, defined by $S_+e_n=\sqrt{n}\;e_{2n}$ for $n\in\N.$
Since $c_{00}\subseteq\cD(S_+^*)$, $S_+$ is closable. Its matrix has the entry $\sqrt n$ at position $(2n,n)$ and zeros elsewhere, so $\fp_\cB(S_+)=+\infty$. On $c_{00}$, the commutator is $ [S_+,P_n]=\sum_{k=\lfloor n/2\rfloor+1}^{n}\sqrt k E_{2k,k}$. Hence $\fp_\cB([S_+,P_n])=n<\infty$, and Proposition~\ref{Prop: propagation} applies with $m=2n$. We have
\[
 \norm{[S_+,P_n]}_u=\sqrt n,\qquad
 \norm{[S_+,P_n]}_1=\sum_{k=\lfloor n/2\rfloor+1}^n\sqrt k,\qquad
 \norm{[S_+,P_n]}_2^2=\sum_{k=\lfloor n/2\rfloor+1}^n k .
\]
Moreover, for $a=1,2$ we obtain
\[
 \frac{\norm{[S_+,P_n]}_a}{\norm{P_n}_a}
 \ge\frac{\lceil n/2\rceil^{1/a}\big(\tfrac{n+1}{2}\big)^{1/2}}{n^{1/a}}
 \ge\frac{(n+1)^{1/2}}{2^{1/a+1/2}}\xrightarrow[n\to\infty]{}\infty ,
\]
and $\cP_\cB$ is neither a $1$-F\o lner nor a $2$-F\o lner sequence for $S_+$.
\end{example}


\section{Quasidiagonality and F\o lner sequences in Quantum Mechanics}

In this section we apply the methods developed before to some relevant quantum mechanical structures. In the next subsection we apply an algebraic version of F\o lner approximations to the Weyl algebra and then consider a Hilbert space representation of this algebra, necessarily in terms of unbounded operators.\footnote{There is a C*-algebraic version of the Weyl algebra also called Weyl C*-algebra or CCR-algebra (see \cite{Manuceau68,LM22}).}
In this way we are able to relate algebraic F\o lner subspaces with F\o lner sequences of projections
and will show in this framework the usual algebraic and analytical difficulties when working with unbounded operators.

\subsection{Algebraic amenability in quantum mechanics}\label{sec:qm}
As mentioned in the introduction, F\o lner sequences appeared naturally when extending amenability notions beyond group theory. In particular, an algebraic version of amenability was introduced by Gromov in Section~1.11 of \cite{Gromov99}. 
We will restrict here to associative algebras of countable dimension over the complex numbers, but this notion can be presented in full generality
(see also \cite[Section~3]{ALLW-1} for a thorough analysis of this concept and additional references).

Algebraic amenability is usually introduced via the existence of finite dimensional subspaces of the algebra 
satisfying a F\o lner-type condition which we express locally.

\begin{definition}\label{def:algebraic-amenability}
Let $\fA$ be an associative algebra of countable dimension over $\C$.
\begin{enumerate}
  \item[(i)] Let $\cF\subset\fA$ be a finite subset and $\varepsilon \geq 0$. A non-zero 
  finite dimensional linear subspace $V \subseteq\fA$ is said to be a 
  \textit{($\cF, \varepsilon$)-F\o lner subspace} if it satisfies
         \begin{equation}\label{eq:algebraic_amenability}
         \frac{\dim(aV+V)}{\dim(V)}\leq 1+\varepsilon\;\;,\quad a\in\cF\; .
         \end{equation}
\item[(ii)] $\fA$ is said to be \textit{algebraically amenable} if for every $\varepsilon >0$ and every finite set $\cF\subset\fA$, there exists a left ($\cF, \varepsilon$)-F\o lner subspace.
\end{enumerate}
\end{definition}

As a first algebraic motivation for the analytical considerations in relation to unbounded operators in Section~\ref{sec:unbounded} we consider here Weyl algebras
which appear as primitive quotients of the enveloping algebra of a nilpotent Lie algebra 
over $\C$ (see, e.g., \cite{dixmier68}). 
For the purpose of this article it is enough to focus on the simplest case with two generators (similar results extend straightforwardly to the case of $2n$ generators).
The Weyl algebra $\cW$ is generated by the symbols $q,p$ satisfying the canonical commutation relation:
\[ 
 \cW:=\C\langle p,q \mid qp-pq=i\1 \rangle\;.
\]
It is easy to see that this algebra is simple and admits a basis of monomials where the $q$'s are written to the right of the $p$'s specifying a normal form for elements in $\cW$.
We show next that the Weyl algebra is also algebraically amenable:

\begin{proposition}\label{Prop:C_algebraic_amenable} The algebra $\cW:=\C\langle p,q \mid qp-pq=i\1 \rangle$ is algebraically amenable.

\end{proposition}
\begin{proof}
Consider the subspaces 
\begin{equation}\label{eq:vn}
 V_n:=\mathrm{span}_\C\left\{p^kq^l \mid k,l \in \N_0,\; k+l\leq n\right\}
 \quadtext{for} n\in\N_0:=\{0,1,\dots\}\,.
\end{equation}
Since the monomials $p^kq^l$ form a basis of the Weyl algebra $\cW$, we have
$\dim V_n=(n+1)(n+2)/2$. Moreover, 
the following inclusion holds $V_nV_m\subset V_{n+m}$ for all $n,m \in \N_0$ since, given $p^kq^l\in V_n$ and $p^{k'}q^{l'}\in V_m$, we have $p^kq^l\cdot p^{k'}q^{l'}\in V_{n+m}$.
Indeed, repeated application of the canonical commutation relation gives
\[q^lp^{k'}=\sum_{s=0}^{\min\{l,k'\}}i^ss!\binom{l}{s}\binom{k'}{s}p^{k'-s}q^{l-s}\in \mathrm{span}_\C\left\{p^iq^j \mid  i+j\le k'+l\right\},\]
and hence, $p^kq^l\cdot p^{k'}q^{l'}\in \mathrm{span}_\C\left\{p^iq^j \mid i+j\le k+k'+l+l'\right\}\subset V_{n+m}$.

To show algebraic amenability (cf.~Definition~\ref{def:algebraic-amenability})
fix $\varepsilon>0$ and a non-empty finite subset $\cF\subset \cW$. Choose
$\delta\in\N_0$ such that $\cF\subset V_\delta$.  By linearity and the observation above we have $aV_n+V_n\subset V_{n+\delta}$ for any $a\in \cF$. It follows that
\[
\frac{\dim(aV_n+V_n)}{\dim V_n}
\le\frac{\dim V_{n+\delta}}{\dim V_n}
=\frac{(n+\delta+1)(n+\delta+2)}{(n+1)(n+2)}
\le\Bigl(1+\frac{\delta}{n+1}\Bigr)^2
\le 1+\varepsilon
\]
whenever $n+1\ge\delta/(\sqrt{1+\varepsilon}-1)$.

In conclusion, $V_n$ is a $(\cF,\varepsilon)$-F\o lner subspace any such $n$.
\end{proof}

We consider next $*$-representations of $*$-algebras in terms of unbounded operators. It is well-known that the canonical commutation relation $qp-pq=i\1$ cannot be represented by bounded operators (see, e.g., \cite{dixmier58,Schmuedgen20}). However, we recall in this section that it does allow a distinguished representation (called Schr\"odinger 
representation) in terms of essentially self-adjoint operators. 
Motivated by the algebraic amenability of $\cW$, we provide a natural generalization of the quasidiagonal and F\o lner type approximations in the context of unbounded operators.

We start by revisiting the notion of representation of a $*$-algebra by unbounded operators (see, e.g.,  \cite[Chapter~8]{Schmuedgen20}).
\begin{definition}
 Let $\mathfrak{A}$
 be a $*$-algebra with involution $a \mapsto a^\dagger$. Denote by $\cD$ a dense subspace of $\cH$ and by $\cL(\cD)\subset\cL(\cH)$ the algebra of linear (possibly unbounded) closable operators of $\cD$ into itself. An algebra homomorphism $\pi\colon\fA\to \cL(\cD)$ is said to be a \textit{{$*$-representation of $\fA$ on $\cD(\pi):=\cD$}} if for each $a\in \fA$ we have that $\pi(a^\dag)$ is the restriction to $\cD$ of the adjoint operator $\pi(a)^*$, that is,
 \begin{equation}\label{eq:*-representation conditon}
   \langle\pi(a) \varphi,\psi\rangle=\langle\varphi,\pi(a^\dag)\psi\rangle \quad\text{for all}\quad \varphi, \psi \in \cD\;.
\end{equation}
\end{definition}

This definition can be applied to the Weyl algebra, which can be turned into an involutive algebra (denoted again by $\cW$) by fixing the involution on the generators via $p^\dag:=p$ and $q^\dag:=q$. This algebra admits a well-known (unbounded) $*$-representation $\pi_S$ on $\cS(\R)\subset L^2(\R)$ in terms of the usual position and momentum operators, called the Schr\"odinger representation of the Weyl algebra (cf. \cite[Example~2.5.2 and 8.3.7]{Schmuedgen20}): taking the Schwartz functions as the common domain 
$\cD(\pi_S):=\cS(\R)$ and defining for any $\psi\in\cS(\R)$
$$
(\pi_S(q)\psi)(x):=(Q\psi)(x)=x\psi(x)\quad \text{and} \quad (\pi_S(p)\psi)(x)=(P\psi)(x)=-i\psi'(x)\;.
$$
Since the Schr\"odinger representation $\pi_S$ is injective, it preserves linear dimension, and the $*$-algebra of operators $\pi_S(\cW)$ is also algebraically amenable, with amenability implemented by a
sequence of finite dimensional subspaces given by $\{\pi_S(V_n)\}_{n\in\N_0}$, where $V_n$ is defined in Eq.~\eqref{eq:vn}.
In order to understand the analytical implications and limitations of quasidiagonal and F\o lner type approximations for operators in relation to algebraic amenability,  
we will use the extensions of the notions given in Definition~\ref{def:foelner} to the unbounded scenario introduced in Definitions~\ref{Def:QDunbounded} and \ref{Def:FOEunbounded}.

There is also a natural relation between algebraic amenability and concrete operator algebras 
having a F\o lner sequence. Consider first the bounded case and let $\fA\subset\cB(\cH)$ be a concrete separable C*-algebra  
and $\fA_0\subset\fA$ a dense $*$-subalgebra with a cyclic vector $\psi_0$. 
If $\fA_0$ is algebraically amenable, then there is a F\o lner sequence for all operators in $\fA$ (see \cite[Theorem~3.17]{ALLW-2} for a precise and more general statement). 
Concretely, there is a natural linear map between the algebra and the Hilbert space given by
\[
  \Phi\colon \fA_0\to\cH\;,\qquad \Phi(a):=a\psi_0\;.
\]

In the unbounded scenario, the same idea relates subspaces of an $*$-algebra $\fA$
with orthogonal projections on a Hilbert space. Let $(\pi\colon\fA\to \cL(D),\cD\subset\cH)$ be a $*$-representation with cyclic vector $\psi_0\in\cD$, and consider the linear map
\begin{equation}\label{eq:map}
    \Phi\colon \fA\to\cH\;,\qquad \Phi(a):=\pi(a)\psi_0\;.
\end{equation}
To a finite-dimensional subspace $V\subset \fA$ (for instance, a F\o lner subspace) we associate the finite dimensional subspace $\Phi(V)=\pi(V)\psi_0\subset \cH$, and we denote by $P_V$ the orthogonal projection onto it. We will show in the next subsection that, in contrast with the bounded case, the projection $P_V$ associated with an algebraic F\o lner subspace $V$ need not be a F\o lner projection for general $*$-representations.

\subsection{Examples in the Schr\"odinger representation of the Weyl algebra}
We conclude this section making contact with the notion of algebraic amenability of the Weyl algebra in its Schr\"odinger representation. The sequence of F\o lner subspaces presented in Section~\ref{sec:qm} will determine a canonical sequence of finite-rank projections on $L^2(\R)$.

Solving the stationary Schr\"odinger equation for the quantum harmonic oscillator with mass $m$ and frequency $\omega$ (and taking $\hbar=1$), we get the following eigenvectors, which form an orthonormal basis of $L^2(\R)$:
\begin{equation}   \label{eq:eigenfunctions}
    \psi_n(x)=\frac{1}{(\sqrt{\pi}\: x_0\:2^n\:n!)^{1/2}}H_n \left(\frac{x}{x_0} \right)
                  \exp\left\{-\frac{1}{2}\left(\frac{x}{x_0}\right)^2\right\}\;,\quad n\in\N_0\,,
\end{equation}
where  $H_n:=(-1)^ne^{x^2}\frac{d^n}{dx^n}e^{-x^2}$ is the $n^{\text{th}}$-order Hermite polynomial (with $H_0:=\mathbbm{1}$) and $x_0:=1/\sqrt{m\omega}$ (see \cite[Theorem~11.4]{bHall13}). In what follows we choose units with $m=\omega=1$, so that $x_0=1$.

Given the orthonormal basis 
$\cB:=\{\psi_n\}_{n\in\N_0}\subset\cS(\R)\subset L^2(\R)\cong\ell_2(\N_0)$ 
the operators $\pi_S(q)$ and $\pi_S(p)$ have the following matrix expressions:
\begin{equation}\label{eq:Def_mat_QP}
    \pi_S(q)=\frac{1}{\sqrt{2}}
    \begin{pmatrix}
      0 & 1 & 0 & 0 & \ldots\\
      1 & 0 & \sqrt{2} & 0 & \ldots\\
      0 & \sqrt{2} & 0 & \sqrt{3} & \ldots \\
      0 & 0 & \sqrt{3} & 0  &\ldots \\
      \vdots & \vdots & \vdots & \vdots & \ddots \\
    \end{pmatrix},
    \quadtext{and}
    \pi_S(p)=\frac{i}{\sqrt{2}}
    \begin{pmatrix}
      0 & -1 & 0 & 0 & \ldots\\
      1 & 0 & -\sqrt{2} & 0 & \ldots\\
      0 & \sqrt{2} & 0 & -\sqrt{3} & \ldots \\
      0 & 0 & \sqrt{3} & 0  &\ldots \\
      \vdots & \vdots & \vdots & \vdots & \ddots \\
    \end{pmatrix}.\hspace{-0.2cm}
\end{equation}
It is often convenient to express them in terms of the creation and annihilation operators  $a^*$ and $a$, respectively, defined as weighted shifts by
\begin{equation}\label{eq:S_w_defintion}
a^*\psi_n=\sqrt{n+1}\psi_{n+1} \quad \text{for all} \quad n\in \N_0, \quad\ \text{and} \quad a\psi_n=\begin{cases}
    \sqrt{n}\psi_{n-1} &n\in \N\\
    0& n=0\;,
\end{cases}
\end{equation}
with domains $\cD(a^*)=\cD(a)=\cS(\R)$. Namely, we have
\begin{equation} \label{eq:Q,P_linearcombinaiton_Sw}
    \pi_S(q)=\frac{1}{\sqrt{2}}(a^*+a)  \quad \text{and} \quad\pi_S(p)=\frac{i}{\sqrt{2}}\left(a^*-a\right),
\end{equation}
and therefore, $\pi_S(\cW)=\C\langle a,a^*\rangle$ with the commutation relation $[a,a^*]=\1$.

Since $\psi_0$ is a cyclic vector for $\pi_S(\cW)$ there is a map $\Phi$ (defined in Eq.~\eqref{eq:map}) 
from the lattice of subspaces of $\cW$ to the lattice of projections on the Hilbert space $\cH=L^2(\R)$. In the concrete case of the F\o lner subspaces $\{V_n\}_{n\in \N_0}$ defined in Eq.~\eqref{eq:vn} we consider first their representation as 
\[
 \pi_S(V_n):=\mathrm{span}_\C\left\{a^k(a^*)^l \mid k,l \in  \N_0, \; k+l\leq n\right\}\;
\]
and the corresponding finite-rank orthogonal projection $P_{n+1}:=P_{V_n}$ onto the subspace $\pi_S(V_n)\psi_0\subset L^2(\R)$.
Note that, by definition of the creation and annihilation operators, $\{P_{n}\}_{n\in\N}$ is associated to the canonical filtration $\{\operatorname{span}\{\psi_0,...,\psi_{n-1}\}\}_{n\in\N}$.
As the sequence of projections $\cP_\cB=\{P_n\}_{n\in\N}$ encodes the algebraic amenability of 
$\cW$, it seems a natural candidate as a F\o lner sequence for (every element in) $\pi_S(\cW)$. However, algebraic F\o lner sequences need not induce F\o lner sequences of projections for unbounded operators, as we show in the next result.
In fact, the sequence of projections $\{P_n\}_{n\in \N}$ associated with the canonical F\o lner subspaces $\{V_n\}_{n\in \N_0}$ of the Weyl algebra is not quasidiagonalizing for either $\pi_S(q)$ or $\pi_S(p)$.
But since these operators are essentially self-adjoint their closures {\em do have a different quasidiagonalizing sequence} based on the procedure explained in Subsection~\ref{subsec:berg-unbounded} that uses the spectral projections on a dyadic partition of the spectrum.

\begin{theorem}\label{Theorem: FolnerSequence_PQ}
Let $\cP_\cB=\{P_n\}_{n\in\N}$ be the sequence of finite-rank projections associated to the canonical filtration defined before. Then
\begin{enumerate} 
    \item $\cP_{\cB}$ is not a quasidiagonalizing sequence for $\pi_S(q)$ nor for $\pi_S(p)$;
    \item $\cP_\cB$ is a 1-F\o lner sequence for $\pi_S(q)$ and $\pi_S(p)$ but not a 2-F\o lner sequence for any of them;
    \item $\cP_\cB$ is neither a 1-F\o lner sequence nor a 2-F\o lner sequence for $\pi_S(p^2)$ nor for $\pi_S(q^2)$.
\end{enumerate}
\end{theorem}
\begin{proof}
Note first that the stability of domains is satisfied since $P_n\cH\subset \cS(\R)$ for all $n\in\N$.

(1) As the position and momentum operators are represented in terms of weighted shifts we have
        $$\norm{[\pi_S(q),P_n]\psi_n}=\norm{[\pi_S(p),P_n]\psi_n}=\sqrt{n/2}$$ 
        and, thus,  $\cP_{\cB}$ is not a quasidiagonalizing sequence for $\pi_S(q)$ nor for $\pi_S(p)$.
        
(2) From Eq.~\eqref{eq:Q,P_linearcombinaiton_Sw} we consider first the following estimates
        for the commutators in $1$- and $2$-norms:
        \begin{align*}
        \norm{[\pi_S(p),P_n]}_1\:&\leq \frac{1}{\sqrt{2}}\left(\norm{[a^*,P_n]}_1+\norm{[a,P_n]}_1\right),
        \quadtext{(same estimate for} \norm{[\pi_S(q),P_n]}_1)\;,\\
        \norm{[\pi_S(q),P_n]}_2&=\norm{[\pi_S(p),P_n]}_2=\frac{1}{\sqrt{2}}\sqrt{\left\Vert[a^*,P_n]\right\Vert_2^2+\big\Vert[a,P_n]\big\Vert_2^2}
        \;.
        \end{align*}
        Note now that $a^*$ acts on the basis $\{\psi_n\}_{n\in\N_0}$ just as the weighted shift $S_w$ from Example~\ref{Ex: S_w} with unbounded weight $w=(1,\sqrt{2},\sqrt{3},\dots)$. 
        Since the conclusions of Example~\ref{Ex: S_w} also apply to $S_w^*$, it follows that $\cP_\cB$ is a $1$-F\o lner sequence for both $a$ and $a^*$ but is not a $2$-F\o lner sequence for either of them. This completes the proof for (2).
                
(3) By definition of the creation and annihilation operators,
squaring \eqref{eq:Q,P_linearcombinaiton_Sw} can be applied to the basis elements
and, using that $a^*a+aa^*$ is diagonal in $\cB$, we obtain
\[
[\pi_S(q^2),P_n]=-[\pi_S(p^2),P_n]=\tfrac12 A_n,
\quadtext{where} A_n:=[(a^*)^2+a^2,P_n]\;.
\]
For $n\ge 2$ we have that $\fp_\cB(A_n)=2$ and a direct computation gives 
\begin{align*}A_n\psi_{n-2}=\sqrt{n(n-1)}\:\psi_n, \quad&A_n\psi_{n-1}=\sqrt{n(n+1)}\:\psi_{n+1}\;,\\
A_n\psi_{n}=-\sqrt{n(n-1)}\,\psi_{n-2},\quad &A_n\psi_{n+1}=-\sqrt{n(n+1)}\,\psi_{n-1}\;.
\end{align*}
Hence, the singular values of $P_{n+2}A_nP_{n+2}$ are $\sqrt{n(n-1)}$ and $\sqrt{n(n+1)}$, each with multiplicity $2$.
By Proposition~\ref{Prop: propagation}, 
\[\frac{\norm{[\pi_S(p^2),P_n]}_1}{\norm{P_n}_1}=\frac{\norm{[\pi_S(q^2),P_n]}_1}{\norm{P_n}_1}
=\sqrt{1-\tfrac1n}+\sqrt{1+\tfrac1n}\mathop{\longrightarrow}_{n\to\infty} 2\;,\]
and
\[\frac{\norm{[\pi_S(p^2),P_n]}_2}{\norm{P_n}_2}=\frac{\norm{[\pi_S(q^2),P_n]}_2}{\norm{P_n}_2}=\sqrt{n}
\mathop{\longrightarrow}_{n\to\infty}\infty\;,\]
so the F\o lner condition is not satisfied in either case. 
\end{proof}

\end{document}